\documentclass[12pt]{amsart}
\usepackage[a4paper,margin=23mm]{geometry}
\usepackage{amsfonts, amsthm, amsmath, amssymb}
\usepackage{hyperref}
\hypersetup{colorlinks=false}

\renewcommand{\leq}{\leqslant}

\renewcommand{\geq}{\geqslant}

\DeclareMathOperator{\lcm}{lcm}

\usepackage{comment}
\usepackage{graphics}
\usepackage{aliascnt}

\usepackage{enumerate}
\usepackage{amsmath}
\usepackage{amsfonts}
\usepackage{amssymb}
\usepackage{amsthm}
\usepackage{comment}
\usepackage{mathtools}
\usepackage{xcolor}

\newtheorem{theorem}{Theorem}[section]
\newtheorem{corollary}[theorem]{Corollary}

\newtheorem{lemma}[theorem]{Lemma}
\newtheorem{proposition}[theorem]{Proposition}
\theoremstyle{definition}
\newtheorem{definition}[theorem]{Definition}
\newtheorem{remark}[theorem]{Remark}

\numberwithin{equation}{section}

\newcommand\RR{\mathbb{R}}

\newcommand\ZZ{\mathbb{Z}}
\newcommand\N{\mathbb{N}}
\newcommand\NN{\mathbb{N}}

\newcommand\CC{\mathbb{C}}

\newcommand\QQ{\mathbb{Q}}

\newcommand\e{\mathrm{e}}

\numberwithin{equation}{section}

\usepackage[backend=bibtex,
style=numeric,
isbn=false,
doi=false
]{biblatex}
\title{Ratio of sum of digits functions in two bases}
\author{Pascal Jelinek}
\email{pascal.jelinek@unileoben.ac.at}
\address{
	Department Mathematics and Information Technology,
	Leoben University of Technology,
	Franz-Josef-Strasse 18, 8700 Leoben, Austria.\
}

\begin{document}
%just 2 bases with (hopefully) correct exponent
%https://math.stackexchange.com/questions/3909783/tail-lower-bounds-using-moment-generating-functions

%bin coef (and also mult coeffs) optimal exponent
%https://math.stackexchange.com/questions/235962/asymptotics-of-binomial-coefficients-and-the-entropy-function
\begin{abstract}
	In 2019 La Bretèche, Stoll and Tennenbaum showed that the ratio of the sum of digits function $s_{q_1}(n)/s_{q_2}(n)$ of two multiplicatively independent bases $q_1$ and $q_2$ is dense in $\QQ^+$. Recently Spiegelhofer proved that in the special case $s_2(n)/s_3(n)=1$ we have infinitely many solutions. Spiegelhofer extended this jointly with Drmota to show that the pair $(s_2(n),s_3(n))$ attains almost every value of $\NN^2$ and hence, in particular, that every rational ratio is attained infinitely many times.\\
	In this paper we show that, indeed, for any pair of multiplicatively independent bases $p$ and $q$, that the ratio attains every rational number infinitely many times. We also study this problem in the multiplicatively dependent case, hence giving a complete characterisation in the case of 2 bases.
\end{abstract}
\maketitle
\tableofcontents
\renewcommand{\thefootnote}{\fnsymbol{footnote}}
% 11A63 Radix representation; digital problems
% 60F10 Large deviations
% 11B25 Arithmetic progressions
% 11B50 Sequences (mod m)
\footnotetext{ The author was supported by the Austrian Science Fund (FWF), Project P36137-N.}
\renewcommand{\thefootnote}{\arabic{footnote}}

\section{Introduction and main result}
Let $q>1$ be an integer. We define the base $q$ expansion for any positive real number $a$ to be the sum
\begin{equation}
	a= \sum_{i\in \ZZ} \delta_i(n)q^i\,.
\end{equation}
We restrict the numbers $\delta_i(n)$ to be integers such that $0\leq \delta_i(n) \leq q-1$ holds for each $i$, furthermore, we do not allow an infinite string of digits all equal to $\mathrel{{q}{-}{1}}$. With this restriction, it can easily be shown that the base $q$ expansion of any positive real number exists and is unique.

One of the most general conjectures about the expansion of a number in two different bases is Furstenberg's conjecture in dynamics, which he proposed in 1970 \cite{Furstenberg1970}. It concerns the dimensions of the shift spaces of an irrational number in two different bases. Let $O_q(x)$ be defined as 
\[ 
    O_q(x)=\{q^{\alpha}x:\alpha\in \NN\}.
\]
Given any irrational real number $x\in [0,1]$, and any two multiplicatively independent integers $q_1,q_2>1$, Furstenberg's conjecture states that
\[\dim(\overline{O_{q_1}(x)})+\dim(\overline{O_{q_2}(x)})\geq1\;,\]
where the dimension in the above inequality is the Hausdorff dimension. Shmerkin \cite{Shmerkin2019} and Wu \cite{Wu2019} proved that the set of exceptions to this conjecture has Hausdorff dimension zero. However, the most interesting case where one of the dimensions is zero is still open.\\

If one restricts to the case of representations of positive integers, one gets representations of the form
\begin{equation}
	a= \sum_{i\in \NN} \delta_i(n)q^i\,.
\end{equation}
Erd\H{o}s \cite{Erdos1979} asked whether each power of 2, except $1,4$ and $256$, has a digit in base $3$ which is equal to $2$. Dimitrov and Hove \cite{DimitrovHowe2021} proved that if there is a counterexample, then it needs to be the sum of at least 25 distinct powers of $3$. Recently, Bloom and Croot \cite{BloomCroot2025} published a preprint on arXiv studying if there are infinitely many integers that have restricted digits. This means that the digits that appear in the base $p$ representation are only allowed to be from a set $A\subsetneq \{0,1,\dots,\mathrel{{p}{-}{1}}\}$. Integers with certain sets of restricted digits are closely related to the divisibility of central binomial coefficients, via the identity
\[
	\nu_p \left(\binom{2n}{n}\right) = \frac{2s_p(n)-s_p(2n)}{p-1}
\]
for any prime number $p$. Therefore, $p$ does not divide $\binom{2n}{n}$ if and only if $n$ has no digit that is larger than $(p-1)/2$ in its base $p$ representation. A long-standing conjecture by Graham \cite{Erdos1980} is that there are infinitely many central binomial coefficients that are not divisible by $105=3\cdot5\cdot7$, which corresponds to integers $n$ that consists only of digits equal to $0,1$ in base $3$, only of digits equal to $0,1,2$ in base $5$, and only of digits equal to $0,1,2,3$ in base 7. A related conjecture by Erd\H{o}s \cite[p.71]{Erdos1980} is that no central binomial coefficient is square-free for $n>5$. The latter was solved by Sark\H{o}zy \cite{Sarkozy1985} for large enough $n$ and by Granville and Ramaré \cite{GranvilleRamare1996} in general.

In this paper, we are concerned with the values of the sum of digits function in different bases. Given the base $q$ representation of an integer, we can define the sum of digits function to be
\begin{equation}
	s_q(n):= \sum_{i\geq 0}\delta_i(n)\,.
\end{equation}
It is easy to show that for $n<N$, we have that the expectation is 
\[
	\mathbb{E}(s_q(n))=\frac{q-1}{2\log q}\log N
\]
and the standard deviation is
\[
	\sigma = \sqrt{\frac{q^2-1}{12\log q}\log N}.
\]
Combining these two facts with the fact that $\max s_q(n)< \frac{q-1}{\log q}\log N$, Chebyshev's inequality implies that for all $q_1<q_2$
\[
	\substack{\lim\\N\rightarrow\infty}\frac{\{n\leq N:s_{q_1}(N)<s_{q_2}(N)\}}{N} = 1.
\] 
On the other hand, there exists infinitely many integers $n$ such that $s_{q_2}(n)<s_{q_1}(n)$, for example if $n=q_2^t$ for some positive integer $t$. However, Senge and Straus \cite{SengeStraus1973} showed in 1973 that the simultaneous inequalities $s_{q_1}(n)<c$ and $s_{q_2}(n)<c$ have only finitely many solutions and Steward \cite{Stewart1980} later proved a quantitative version of this theorem. It remained open whether it happens infinitely many times that $s_{q_1}(n)$ and $s_{q_2}(n)$ have the ``same magnitude'' or are in fact equal.

This question was first approached in a paper by Deshouillers, Habsieger, Landreau and Laishram \cite{DHLL2017} in 2017. The wanted to find the smallest number $c$ such that the inequality
\[
	|s_3(n)-s_2(n)|<c\log n
\]
has infinitely many solutions. It is a well-known result that 
\[
    \mathbb{E}(s_3(n))-\mathbb{E}(s_2(n))\approx0.18889\log n,
\]
and hence it is easy to see that there are infinitely many solutions for $c=0.18889$. The authors of \cite{DHLL2017} showed that one can take $c=0.1457205$ and still satisfy the condition: 
\begin{theorem}[Deshouillers, Habsieger, Landreau, Laishram, 2017] \label{DHLL17}
	For sufficiently large $N$, we have that
	\[
	\#\bigl\{n\leq N:\lvert s_3(n)-s_2(n)\rvert\leq 0.1457205\log n\bigr\}>N^{0.97}.
	\]
\end{theorem}
In 2019, de la Bretèche, Stoll and Tennenbaum \cite{BretecheStollTenenbaum2019} improved this theorem by showing that for any two multiplicatively independent integers $q_1,q_2>1$ the ratio
	\[
        \{s_{q_1}(n)/s_{q_2}(n):n\geq 1\}
	\]
is dense in $\mathbb{R}^+$. In particular, this shows that the constant in Theorem \ref{DHLL17} can be replaced by any $\varepsilon>0$. The question whether there are infinitely many solutions to $s_{q_1}(n)=s_{q_2}(n)$ or not was answered by Spiegelhofer \cite{Spiegelhofer2023a} in 2023 for $q_1=2$ and $q_2=3$:
\begin{theorem}[Spiegelhofer, 2023] \label{bases23}
	There exist infinitely many $n$ such that
	\[
		s_2(n)=s_3(n).
	\]
	In particular, we have that
	\[
	\#\bigl\{n\leq N: s_3(n)=s_2(n)\bigr\}\gg N^{0.79}.
	\]
\end{theorem}
In 2025, Drmota and Spiegelhofer \cite{DrmotaSpiegelhofer2025} proved estimates for the number of solutions of $s_2(n)=a,s_3(n)=b$ for any $(a,b)\in \NN^2$. These estimates imply that almost all pairs $(a,b)\in \NN^2$ are attained.

In this paper, we extend Theorem \ref{bases23} by Spiegelhofer \cite{Spiegelhofer2023a} to the case of any bases $q_1,q_2$ that are multiplicatively independent. Additionally, we show that any rational number is attained by $s_{q_1}(n)/s_{q_2}(n)$ for infinitely many $n$, which in turn improves on the result by de la Bretèche, Stoll and Tennenbaum, who showed that every positive rational number is a limit point.\\
\begin{theorem}\label{main_thm_indep}
	Let $1<q_1,q_2$ be two multiplicatively independent integers. For any  $r\in \QQ^+$ we have that there exist infinitely many $n\in \NN$ such that
	\[
		\frac{s_{q_1}(n)}{s_{q_2}(n)}=r.
	\]
	For the quantitative statement, we need to assume that $q_1,q_2, r$ are chosen such that
	\[
		r \frac {(q_2-1)\log q_1}{(q_1-1)\log q_2} <1.
	\]
	This can always be achieved by exchanging $q_1$ and $q_2$, and replacing $r$ by $1/r$. With this assumption, we define $\beta$ to be 
	\[
		\beta:=\substack{\lim\\N\rightarrow\infty}\frac{\{0\leq n \leq N: s_{q_2}(n)= \frac{q_1-1}{2r\log q_1}\log N\}}{\log N}.
	\]
	Then we have that
	\[
		\#\{1<n\leq N: \frac{s_{q_1}(n)}{s_{q_2}(n)}=r\} \gg N^{\beta-\varepsilon}.
	\]
\end{theorem}
\begin{remark}
    For the readers familiar with the paper by Spiegelhofer, we note that his method directly generalises to the case where $q_1$ and $q_2$ are coprime and $r \frac {(q_2-1)\log q_1}{(q_1-1)\log q_2}<\frac 12$. The ratio $1/2$ corresponds to the study of arithmetic progressions, where the number of steps $N$ and the step size $d$ satisfy the relationship $d=(N)^{1/2}$. Our main contribution is in providing new methods to study arithmetic progressions of arbitrarily large step size.
\end{remark}

In addition to studying the ratios of the sum of digits function in the case of two multiplicatively independent bases, we provide below the case where $q_1$ and $q_2$ are multiplicatively dependent. Hence, together with Theorem \ref{main_thm_indep} above, this gives a complete characterisation of the possible ratios.
\begin{theorem}\label{main_thm_dep}
	Let $1<q_1,q_2$ be two distinct multiplicatively dependent numbers, say $q_1=b^k$ and $q_2=b^{\ell}$ for some coprime natural numbers $k,\ell$  and let $r\in \QQ^+$. Then we have that there are infinitely many $n\in \NN$ such that
	\[ 
		\frac{s_{q_1}(n)}{s_{q_2}(n)}=r
	\]
	as long as
	\[ 
		\frac{1}{b^{\ell-1}}\leq r \leq b^{k-1}.
	\]
More precisely, we have that
	\[
		\#\left\{1<n\leq N: \frac{s_{q_1}(n)}{s_{q_2}(n)}=r\right\} \gg N^{c},
	\]
	where $0<c\leq 1$ is an effectively computable constant, which depends on $k$ and $\ell$ 
\end{theorem}

\subsection{Notation}
We fix some notation for the remainder of the paper. We use $\mathrm{e}(x):=\exp (2\pi i x)$. By $\mathcal{C}$ we denote the complex conjugation operator and  $\mathcal{C}^k$ means we apply this operator $k$ times. In addition, $||x||$ denotes the distance to the nearest integer of $x$. We also use the standard asymptotic notations $\ll, \gg, O(\cdot)$ and $o(\cdot)$.\\
$q_1,q_2\geq 2$ are two multiplicatively independent integers, which will also be the two bases in Theorem \ref{main_thm_indep}. We reserve the letter $p$ for prime numbers.\\
We recall that 
\[
	s_q(n) = \sum_{i\geq0}\delta_i(n)
\] 
is the sum of digits function in base $q$ of $n$. Given an interval $I \subset \NN$, we denote by
\[
	s_q^I(n) = \sum_{i\in I} \delta_i(n)
\]
the restriction of the sum of digits function to the digits in the interval $I$. Furthermore, we use $(n)^I$ to denote the digits of $n$ in the interval $I$, i.e.
\[
    (n)^I = \sum_{i\in I} \delta_i(n)q^{i-k},
\]
where $k$ is the minimum integer in the index set $I$.\\
Throughout the paper, we let $N\in \NN$ be a sufficiently large number and let $\varepsilon>0$ be a sufficiently small number. While studying the values of $f(n):=as_p(n)-bs_q(n)$ for $n\in [N,2N]$, we use the following constants, which depend only on $N$ and $\varepsilon$, repeatedly:
\begin{enumerate}
	\item $\lambda = \lfloor \log N \rfloor$ corresponds to the size of the values of $f(n)$
	\item $m = ab(q_1-1)(q_2-1)\left\lfloor\frac{\sqrt{\log N}}{ab(q_1-1)(q_2-1)\sqrt{J}}\right\rfloor+1$ is the modulus
	\item $J=\lfloor \varepsilon (\log \log N)^{1+\varepsilon} \rfloor$
	\item $\beta= \lambda^{3/5}$
\end{enumerate}

\section{Structure of this paper}
The proofs of Theorem \ref{main_thm_indep} and Theorem \ref{main_thm_dep} are independent of each other. Section \ref{sec_Prelim} contains a collection of lemmas which are needed to prove Theorem \ref{main_thm_indep}. In Sections \ref{sec_Proof_of_claims} to \ref{sec_uniform} we prove the three key propositions which we outline below and in Section \ref{sec_proof_main_theorem} we combine them to prove Theorem \ref{main_thm_indep}. Finally, in Section \ref{sec_Proof_of_dep} we prove Theorem \ref{main_thm_dep}.

\subsection{Overview of the proof of Theorem \ref{main_thm_indep}}

Our proof of Theorem \ref{main_thm_indep} follows a similar procedure to the one used by Spiegelhofer \cite{Spiegelhofer2023a} in 2023. He studied the difference $s_2(n)-s_3(n)$ and proved three different properties about this function, which could then be combined to deduce the theorem. Similarly, we study the function $f(n)=as_p(n)-bs_q(n)$ and prove the following three properties:
\begin{enumerate}
	\item There exists an arithmetic progression with step size $\sim N^{\varepsilon}$ and a set $S=\{jm:-J\leq j\leq J\}$ such that if $f(n)=s\in S$, we can find a $d_s\in \NN$ such that $f(n+d_s)=0$.
	\item There exists an arithmetic progression such that along this arithmetic progression $\mathbb{E}(f(n))=0$.
	\item Along the arithmetic progression above, the function $f(n)$ is equidistributed modulo $m$.
\end{enumerate}

First, we observe the identity $s_q(n)\equiv n \bmod q-1$. Denote $r= \gcd(a(q_1-1),b(q_2-1))$. Then we have that
\[
	f(n+d_j)-f(n) \equiv (a-b)d_j \bmod r.
\]
Furthermore, $m \equiv 1 \bmod r$. Hence, a necessary condition such that the equation
\[
	f(n+d_j)-f(n)=jm
\]
is soluble is that $\gcd(a-b,r)\mid \gcd(j,r)$. We have the following proposition, which shows that this is also a sufficient condition.
\begin{proposition}\label{prop_shift}
	Let $\xi_1, \xi_2\leq c(2J+1)\beta$, for some constant $c$. There exists a $d_j<q_1^{\xi_1}$ for all $-J\leq j\leq J$ and there exists an $L' \in \{0,\dots,\lcm(q_1^{\xi_1},q_2^{\xi_2})-1\}$ such that for all $j$ satisfying $-J\leq j\leq J$ and $\gcd(a-b,r)\mid \gcd(j,r)$ the equation
	\[
		f(n+d_j)-f(n)=jm
	\]
    holds for all $n\in A':=L'+q_1^{\xi_1}q_2^{\xi_2}\NN$.
\end{proposition}

Instead of proving properties two and three directly, we prove more general statements about the sum of digits functions in arithmetic progressions. Heuristically, we expect that arithmetic progressions equidistribute modulo $m$, but to find the expected value is more nuanced.
Let $d$ be an integer and let $d'=d/q^e$, where $e=\nu_q(d)$. We want to find the expected value of the sum of digits function for numbers of the form
\[
    \{nd:n<N\}.
\]
We see that the last $e$ digits of each number in the set above are 0. However, for the remaining digits, we expect that there is no bias towards larger or smaller digits and hence, they are on average $(q-1)/2$ if we consider the base $q$ expansion. Therefore, we expect that
\begin{equation}
	\mathbb{E}[s_q(nd)] = \frac{q-1}{2}(\log_q(Nd)-e) = \frac{q-1}{2}(\log_q(N) + \log_q(d')).
\end{equation}
More generally if we consider numbers of the form
\[
\{a+nd:n<N\},
\]
this heuristic gives that
\begin{equation}
	\mathbb{E}[s_q(nd+a)] = \frac{q-1}{2}\big(\log_q(N) + \log_q(d')\big) +s_q^{[0,e)}(a).
\end{equation}

We prove that the heuristic holds true for the following class of positive integers $d$:

\begin{definition}[$N$-admissibility]\label{def_admissibility}
	Let $d>0$ and $q>1$ be integers. Let $e$ be the maximal integer such that $d=q^ed'$. We say that $d$ is $N$-admissible in base $q$ if for $1\leq A \leq N$, we have for all $\kappa\in \NN$ such that $q^{\kappa}\leq d'$
	\[
	\left|\left|\frac{Ad'}{q^{\kappa}}\right|\right|\geq N^{-3}.
	\]
\end{definition}
\begin{remark}
	This definition ensures that we do not have any block of just digits $0$ or of just digits $\mathrel{{q}{-}{1}}$ of length larger than $2\log_q(N)$ in the digit expansion of $Ad$.
\end{remark}
We now show that this condition is generic.
\begin{lemma}
	 Let $N\in \NN$ and $H$ be an arbitrarily large constant. We have that the number of integers $d<N^H$ that are not $N$-admissible is at most 
	 \[
	 	O\left(N^{H-1}\log(N^H)^2\right).
	 \]
\end{lemma}
\begin{proof}
	By the above remark, we see that a number is not $N$-admissible if and only if it contains a string of $2\log N$ many consecutive $0$ or $q-1$. Let $\kappa$ be the minimal integer such that $N<q^{\kappa}$. We consider all integers $d$ in the interval $N^H/q<d<N^H$ as a string of $\kappa$ digits over the alphabet $\{0,\dots, q-1\}$. 
	Let $1\leq i \leq \kappa$ be an integer. We see that the probability that $d$ has only $0$ or has only $q-1$ at all the positions $i+1,\dots,i+2\kappa$ is
	\[
	\frac 2{q^{3 \kappa}} < \frac 2{N^{3}}
	\]
	by observing that there are $q^{2 \kappa}$ many different strings of length $2 \kappa$, only 2 of which are allowed. Similarly, the probability that $Ad$ as only $0$ or has only $q-1$ at all the positions $i+1,\dots,i+3\kappa$ is
	\[
		\frac {2A}{q^{3 \kappa}} < \frac {2A}{N^{3}}
	\]	
	by observing that the strings are not allowed to agree with the first $3\kappa$ digits of  $x/A$ for some $1\leq x\leq A$. Since, this holds for every $A$ and for every $i\leq H\kappa$, we have that at most
	\[
		\frac {2H\kappa}{N^{3}} \sum_{A\leq N}A \leq \frac{2H\kappa}{N}
	\]
	many $d\in [N^H/q,N^H]$ that are not $N$-admissible. This analysis can be done for any interval $[N^H/q^a,N^H/q^{a-1}]$ and we get by the union bound that there are at most
	\[
		N^H\sum_{a=0}^{\log_qN^H} \frac {2(H\kappa-a)}{N} < N^{H-1}\log_qN^H\cdot2H\kappa \ll N^{H-1}\log(N^H)^2 = o(N^H)
	\]
	many $d<N^H$ that are not $N$-admissible.
\end{proof}

Having now defined $N$-admissibility, we can state the next two propositions.

\begin{proposition}[Concentration around the expected value]\label{thm_concentration}
	Let $H$ be any positive constant, and let $q>1$ be a positive integer. Let $\eta<1/5$ a positive real number. Further let $N\in \NN$ be sufficiently large. Let $d<N^H$ be an $N^{\eta}$-admissible integer. Then for all $N<k\leq2N$ except at most $O\left(N\log(N)^{-A}\right)$-many the following inequality holds.
	\[
		\left|s_q(a+dk)-\frac{q-1}2\log_{q}\left(N\right) - \frac{q-1}2\log_{q}(d')-s_q^{[0,e)}(a)\right|\leq Jm,
	\]
	where $d=q^ed'$ as above.
\end{proposition}
\begin{remark}
	Note that, as long as $d<N^{1/2}$, we can even drop the admissibility assumption.
\end{remark}

\begin{proposition}[Equidistribution modulo $m$]\label{prop_equi}
	Let $H$ be any positive constant  and let $\eta>0$ be a sufficiently small constant. Further, let $N$ be a sufficiently large integer and let $q_1,\dots, q_t$ be a collection of $t$ multiplicatively independent integers. Let $d_i$, $a_i$, $\ell_i$, $m_i$ each be a collection of $t$ positive integers, such that for each $1\leq i \leq t$
	\begin{enumerate}
		\item $d_i<N^H$ is $N^{\eta}$-admissible in base $q_i$,
		\item $a_i<d_i$,
		\item $m_i \leq (\log N)^{1/2}/(\log\log N)^{1/2+\varepsilon}$ and $\gcd(q_i-1,m_i)=1$,
		\item $\ell_i<m_i$.
	\end{enumerate}
	Then we have that
	\begin{equation}\label{eq_prop_size}
		\#\left\{ n\in[N,2N]: s_{q_i}(d_in+a_i)\equiv \ell_i\bmod m_i \; \mathrm{for}\; \mathrm{ all } \; 1\leq i \leq t
		\right\} = (1+o(1))\frac{N}{\prod_{i=1}^{t}m_i}.
	\end{equation}
\end{proposition}
\begin{remark}
	It can be shown that $\eta<\min(1/5,1/t)$ is sufficiently small.
\end{remark}

In Section \ref{sec_proof_main_theorem} we find an arithmetic progression $dn+L$ along which $\mathbb{E}(f(n))=0$ and prove that $d$ is $N^{\eta}$-admissible for every $\eta>0$ and for sufficiently large $N$. We then combine the three propositions above to conclude the proof of Theorem \ref{main_thm_indep}.
\section{Preliminaries} \label{sec_Prelim}
This section is a collection of results and methods which are used at various stages throughout the paper. The first lemma below will be a crucial input in the proof of Proposition \ref{prop_shift}. 
\begin{lemma}\label{Addition}
	Given any base $q$ and some fixed $k$, the following holds for any nonzero numbers $d_1, \dots, d_k<q^{\nu-1}$:
	\[
	\#\{0\leq n < q^{\nu-1}: |s_q(n)-s_q(n+d_i)|<\lambda^{1/2} \text{ for all } i\leq k\}\geq q^{\nu-1}\left(1-O\left(\frac{qk\nu}{\lambda}\right)\right)
	\]
\end{lemma}
\begin{proof}
	First observe that since $n,d_1, \dots, d_k<q^{\nu-1}$, we have the following identity  for all $i$:
	\[
	s_q(q^{\nu}m+n)-s_q(q^{\nu}m+n+d_i)=s_q^{(\nu)}(n)-s_q^{(\nu)}(n+d_i).
	\]
	Therefore, it will be enough to study $s_q^{(\nu)}(\cdot)$. We proceed via analysing the probability mass function $\varphi(\_,t,L)$:
	\begin{equation}
		\varphi(j,t,L):=\frac{1}{q^{L}}\#\{0\leq n < q^{L}:s_{q}^{(L)}(n+t)-s_{q}^{(L)}(n)=j\}
	\end{equation}
	and the characteristic function
	\begin{equation*}
		\omega_t(\theta, L):=\sum_{j\in\mathbb{Z}}\varphi(j,t,L)\e(j\theta)
	\end{equation*}
	of the change of the sum of digits function under addition.
	It is easy to prove the following recursions for the probability mass function for any $0\leq a<q$:
	\begin{align}
		\varphi(j,a,L)&=\begin{cases}
			\frac{q-a}{q} & \text{if } j=a\\
			\frac{a(q-1)}{q^{k+1}}	& \text{if } j=a+k(1-q) \text{ for } 1\leq k \leq L-1\\
			\frac{a}{q^L}	& \text{if } j= (L-1)(1-q)+a-q\\
			0 & \text{otherwise }
		\end{cases}\\
		\varphi(j,qt,L+1)&=\varphi(j,t,L)\\
		\varphi(j,qt+a,L+1)&=\frac{q-a}{q}\varphi(j-a,t,L)+\frac{a}{q}\varphi(j+q-a,t+1,L)
	\end{align}
    Similarly, for the characteristic function we have that:
	\begin{align}
		|\omega_t(\theta,L)|&\leq 1\\
		\omega_{qt}(\theta,L+1)&=\omega_t(\theta,L)\\
		\omega_{qt+a}(\theta,L+1)&=\frac{(q-a)\e(a\theta)}{q}\omega_t(\theta,L)+\frac{a\e((q-a)\theta)}{q}\omega_{t+1}(\theta,L)
	\end{align}
	Furthermore, we can obtain the $k$th-moment $m_k$ of the probability mass function via the identity
	\begin{equation}
		m_k(t,L)=\frac{k!}{(2\pi i)^k}[\theta^k]\omega_t(\theta,L).
	\end{equation}
    We now want to find $m_0(t,L)$, $m_1(t,L)$ and $m_2(t,L)$. By observation, we have that $m_0(t,L)= 1$, $m_1(t,L)=0$ (since the expectation of the probability mass function is 0). For the second moment we have for $a<q$:
	\begin{align*}
		m_2(a,L)&=(q-a)a^2/q+\sum_{k=1}^{L-1}a(q-1)(a+k(1-q))^2/q^{k+1}+a((L-1)(1-q)+a-q)^2/q^L\\
		&\leq 4q^4
	\end{align*}
Furthermore, using the recurrence relation we get that
	\begin{align*}
		m_k(qt,L+1)&=m_k(t,L)\\
		m_2(qt+a,L+1)&=\frac{(q-a)m_2(t,L)+am_2(t+1,L)}{q}+a(q-a).
	\end{align*}
Hence, we can bound the variance above by
	\begin{equation}
		m_2(t,\nu)\leq 4q^4\nu.
	\end{equation}
	By Chebyshev's inequality, we get that the number of integers $a\in\{0,\dots,q^{\nu}-1\}$ such that
	\begin{equation}\label{eq_Chebyshev's}
		\left|s_{q}^{(\nu)}(a+d_j)-s_{q}^{(\nu)}(a)\right|\leq R_{q}(4q^4\nu)^{1/2}
	\end{equation}
	is bounded below by $q^{\nu}(1-1/R_q^2)$.\\
	Intersecting $k$ sets, each of them corresponding to one $d_j$, we obtain that the number of $n$ that satisfy equation \ref{eq_Chebyshev's} for each $d_j$ is bounded below by $ q^{\nu}(1-k/R_2^2)$. We set $R_q^2=\lambda/(4q^4\nu)$ and get
	\begin{equation}
		\left|s_{q}^{(\nu)}(a+d_j)-s_{q}^{(\nu)}(a)\right|\leq \lambda^{1/2}
	\end{equation}
	for at least
	\begin{equation}
		q^{\nu}\left(1-O\left(\frac{k\nu}{\lambda}\right)\right).
	\end{equation}
	Therefore, if we restrict $n$ to $n<q^{\nu-1}$, we get that
	\begin{equation}
		q^{\nu-1}\left(1-O\left(\frac{qk\nu}{\lambda}\right)\right)
	\end{equation}
	many $n$ satisfy the desired equation, which proves the lemma.
\end{proof}

Next we prove some general lemmas and properties of $N$-admissible numbers. By applying Schlickewei's $p$-adic subspace theorem, we show that sufficiently large powers of $p$ in base $q$ are indeed also $N$-admissible.

Let us recall the $p$-adic subspace theorem by Schlickewei.
\begin{lemma}[$p$-adic subspace theorem]\label{p-adic-Subspace}
    Let $r\geq N\geq 2$, $C>0$, $\delta>0$ and $S=\{\infty,p_1,\dots,p_r\}$, with $p_i$ distinct primes. Further let $L_{1,\infty},\dots, L_{r,\infty}$ be a linear form in $X_1, \dots, X_n$ with algebraic coefficients in $\CC$ in general positions, and for $1\leq j\leq r$, let $L_{1,p_j},\dots, L_{r,p_j}$ be linear forms in $X_1,\dots, X_n$ with algebraic coefficients in $\overline{\QQ_{p_j}}$ in general position.\\
    Then all integer solutions $x=(x_1,\dots,x_n)$ with $\gcd(x_1,\dots,x_n)=1$ of the inequality
    \[
        \prod_{p\in S}\left|L_{1,p}(x)\dots L_{r,p}(x)\right|\leq C \left|\left| x \right|\right|_{\infty}^{r-n-\delta}
    \]
    are contained in the union of finitely many linear subspaces of $\QQ^n$.
\end{lemma}

Let $L_q(n)$ be the function that returns the maximum between the following two quantities
\begin{enumerate}
	\item the length of the longest block of digits $\mathrel{{q}{-}{1}}$ in the $q$-adic digit expansion of $n$, and
	\item the length of the longest block of digits $0$ in the $q$-adic digit expansion of $n$ that cannot be extended to contain the digit at position zero.
\end{enumerate} 

Now we can say the following about the digit expansion of $q_2^K$ in base $q_1$.
\begin{lemma}\label{no-long-subsequences}
    Let $\eta>0$ be arbitrary. Then for $q_1$ and $q_2$ multiplicatively independent,
    \[
        \sup\limits_{1\leq A \leq {q_1}^{\eta K}} L_{q_1}(Aq_2^K)\leq 2\eta K + o(K)
    \]
    as $K$ tends to infinity. In particular, we have for $\varepsilon>0$, $\eta >0$ and $K$ sufficiently large that
        \[
            \sup\limits_{1\leq A \leq q_1^{\eta K}} L_{q_1}(Aq_2^K) \leq 2\eta K.
        \]
\end{lemma}
\begin{proof}
	First we note that we can write 
    \[
    q_2^K=cq_1^x\tilde{q_{2}}^{K_1}
    \]
    for some integers $x,\tilde{q_{2}},K_1$ and $c$ that satisfy that $x\geq 0$ is maximal, $c\mid q_1$ and $\tilde{q_{2}}\geq 2$. Additionally, we can assume that there exists a prime $p$ such that $p| q_1$ but $p\nmid \tilde{q_{2}}$. Furthermore, because $q_1$ and $q_2$ are multiplicatively independent, there exists some $\alpha>0$ such that $K_1\sim \alpha K$. Since we have chosen $x$ maximal, we know that $q_1\nmid c\tilde{q_{2}}^{K_1}$. Hence, we know that $q_2^K$ ends with $x$ many zeros in base $q_1$ and it suffices to study $L_{q_1}\left(c\tilde{q_{2}}^{K_1}\right)$.\\
	Let $L$ be the size of the largest $0$-block (or $(\mathrel{{q_1}{-}{1}})$-block). We are interested in bounding L from above.\\
	Note that
	\[
		Ac\tilde{q_{2}}^{K_1}=A'\tilde{q_{2}}^{K_1}q_1^{k+L}b\pm a
	\]
	for some $0\leq a \leq q_1^k$ and $0\leq b \leq A'\tilde{q_{2}}^{K_1}q_1^{-k-L}$. Hence, we have
	\[
		\left|A'\tilde{q_{2}}^{K_1}-q_1^{k+L}b\right|\leq q_1^k.
	\]
	To find a lower bound, we use the $p$-adic subspace theorem. Let $r=3$, $C=1$, $\delta>0$, $S=\{p:p|q_1\tilde{q_{2}}\}\cup \infty$, and let $L_{1,p} = x_1$, $L_{2,p}=x_2$, $L_{3,p}=x_1+x_2$. Then we get that
	\[
		\prod_{p\in S}\left|x_1\right|\left|x_2\right|\left| x_1+x_2 \right|\geq \max(|x_1|,|x_2|)^{1-\delta}.
	\]
	Plugging in $x_1=A'\tilde{q_{2}}^{K_1}$ and $x_2=-q_1^{k+L}b$ and using that
	\[
		\prod_{p\in S}|\tilde{q_{2}}q_1|_p=1
	\]
	and
	\[
		\prod_{p\in S}|x|_p \leq |x|
	\]
	we get that either
	\begin{align*}
	    |A'b|\left|A'\tilde{q_{2}}^{K_1}-q_1^{k+L}b\right|&\geq \prod_{p\in S}\left|A'\tilde{q_{2}}^{K_1}\right|\left|-q_1^{k+L}b\right|\left| A'\tilde{q_{2}}^{K_1}-q_1^{k+L}b \right|\\
        &\geq \max(A'\tilde{q_{2}}^{K_1},q_1^{k+L}b)^{1-\delta}\\
        &\geq \left(q_1^{k+L}\right)^{1-\delta}
	\end{align*}
	or $A'\tilde{q_{2}}^{K_1}$ and $-q_1^{k+L}b$ lie in one of finitely many subspaces of the form
	\[
		c_1x_1+c_2x_2 = 0.
	\]
	Since we know that there exists a prime $p$ such that $p| q_1$ and $p\nmid \tilde{q_{2}}$, we see that
	\[
		c_1 \equiv 0 \mod p^{k+L},
	\]
	which is impossible for $L$ large enough.
	Hence, we have that 
	\[
		\left| A'\tilde{q_{2}}^{K_1}-q_1^{k+L}b\right| \geq	\frac{\left(q_1^{k+L}\right)^{1-\delta}}{A'b}.
	\]
	Combining this with both bounds for $\left|A'\tilde{q_{2}}^{K_1}-q_1^{k+L}b\right|$, we get the inequality
	\[
		q_1^k\geq \frac{\left(q_1^{k+L}b\right)^{1-\delta}}{A'b}.
	\]
	Further using that $q_1^kb\leq A'\tilde{q_{2}}^{K_1}$, $A'\leq q_1^{K\eta}$, $K_1=\alpha K$ we get that
	\[
		L \leq \left(\frac{\eta}{1-\delta}+\frac{\left(\eta\log q_1 + \alpha\log \tilde{q_{2}}\right)\delta}{1-\delta}\right)K
	\]
	for all $\delta>0$. Since $\delta$ is arbitrary, we have that
	\[
		L \leq 2\eta K + o(K),
	\]
	which completes the proof.
\end{proof}
We immediately get the following corollary, using only the definition of $N$-admissibility and $L_{q}$.
\begin{corollary}\label{cor_powers_admissible}
	Let $H>0$ be an arbitrary constant. If $N$ is sufficiently large, then we have that all $q_2^K<N^H$ are $N$-admissible in base $q_1$.
\end{corollary}

\begin{lemma}\label{odd-elimination}
	Let $d$ be an $N$-admissible number in base $q$. For all non-negative integers $a,b,s,t$ such that
	\[
		\log_q(N) <a<\log_{q}(d)
	\]
	and
	\[
		0< b-a < \log_q(N) - \lceil\log_{q}(t)\rceil
	\]
	and for all $\omega\in\{0,\dots,q^{b-a}-1\}$, there exists $A\in s + t\NN$ such that $0<A\leq q^3t^2q^{2(b-a)}N^3$ and
	\[
		(Ad)^{[a,b)}=\omega.
	\]
	Furthermore, there exist two constants $c_1, c_2>0$ such that for all $\omega\in\{0,\dots,q^{b-a}-1\}$ and for all $M>q^3q^{2(b-a)}N^3$,
	\begin{equation}
		c_1\frac{M}{q^{b-a}}<\#\{n<M|(nd)^{[a,b)}=\omega\}<c_2\frac{M}{q^{b-a}}.
	\end{equation}
\end{lemma}
\begin{remark}
	The first part of the proof is analogous to \cite[Corollary 4.1]{DrmotaSpiegelhofer2025} (``Odd elimination lemma'').
\end{remark}
\begin{proof}
	Let $m=\max(1,\lceil\log_q(t)\rceil)$. By Dirichlet's approximation theorem there exists a positive integer $C<q^{b-a+m}$ such that
	\[
		\left|\left|\frac{Cd}{q^b}\right|\right| < q^{-(b-a+m)}.
	\]
	By the admissibility assumption we know that
	\[
		\left|\left|\frac{Cd}{q^b}\right|\right| > N^{-3}.
	\]
	This means for the digits of $Cd$ that either $(Cd)^{[a-m,b)}=0$ or $(Cd)^{[a-m,b)}=q^{b-a+m}-1$. Therefore, the digits at position $a-m$ to $b$ are either all 0 or all these digits are equal to $\mathrel{{q}{-}{1}}$. In the first case, the second condition guarantees that $(Cd)^{[b-2\lceil\log_q(N)\rceil,b)}\neq 0$. Hence, we can find a positive integer $c< 3\log_q(N)$ such that
	\[
		(q^cCd)^{[a-m,b)} = x,
	\]
	for some $x\in \{1,\dots,\mathrel{{q}{-}{1}}\}$. Hence, we have that
	\[
		(tq^cCd)^{[a-2,b)} = x',
	\]
	for some $x'\in \{1,\dots,q^2-1\}$. We now consider the values of
	\[
		(rtq^cCd)^{[a,b)}
	\]
	for $r\in \NN$ and notice that as $r$ increases, $(rtq^cCd)^{[a,b)}$ permutes through all $q^{b-a}$ possible values of $(rtq^cCd)^{[a,b)}$.\\
	In the other case where we assume that $(Cd)^{[a-m,b)}=q^{b-a+m}-1$ the second condition guarantees that $(Cd)^{[b-2\lceil\log_q(N)\rceil,b)}\neq q^{2\lceil\log_q(N)\rceil}-1$, and we can find a positive integer $c$ analogously to the first case. Now we consider $(rtq^cCd)^{[a,b)}$ as $r$ increases and notice that we run through all the possible values of $(rtq^cCd)^{[a,b)}$ in the opposite direction. In both cases we see that each value is attained at most $q^2$ many times.
	
	To conclude the poof of the first statement, we note that since we permute through all possibilities we can find an $r$ such that
	\[
		(rtq^cCd+sd)^{[a,b)} = \omega.
	\]
	We choose $A=rtq^cC+s$ and notice that we can choose $r<q^{b-a+2}$ and have that $q^c<N^3$ $C<q^{b-a+m}<qtq^{b-a}$ and hence we have that $A<q^3t^2q^{2(b-a)}N^3$. 
	
	To show the second part, we set $t=1$. We write $n= rq^cC + s$, where $q^cC$ are as above. We fix $s$ and let $r$ vary. Let $r_0$ be the maximum $r$ such that $rq^cC + s<M$. Let us assume that 
	\[
		(q^cCd)^{[a,b)} = 0.
	\]	
	The case were $(q^cCd)^{[a,b)} = q^{b-a}-1$ can be shown analogously. We set $x$ to be
	\[
		x:= (q^cCd)^{[0,b)}.
	\]
	Notice that as we increment $r$, the value of $(rq^cCd)^{[0,b)}$ increases by at most 1 in each step, until
	\[
		\frac{rx}{q^b}>1
	\]
	At this value of $r$, $(q^cCd)^{[a,b)}$ will reach $0$ again. Therefore, there ar/e between
	\[
		\frac{r_0x}{q^b} - 1 \text{ and } \frac{r_0x}{q^b} + 1
	\]
	many such cycles through all possibilities. We also observe that we stay at one value at most $q^2$ many times. This holds for every residue class $s$, hence, we hit the number $\omega$ at most
	\[
		q^cC\left(\frac{r_0x}{q^b} + 1\right)\cdot\left(q^2\right) = q^2\frac{q^cCr_0}{q^{b-a}} + q^cCq^2 < c_2 \frac{M}{q^{b-a}}
	\]
	many times, which can be seen by using that $x<q^a$. Similarly, we get the lower bound:
	\[
		q^cC\left(\frac{r_0x}{q^b}-1\right) = \frac{q^cCr_0}{q^{b-a}} -1 > c_2 \frac{M}{q^{b-a}}.
	\]
\end{proof}

The lemma below provides an upper bound on the number of integers whose sum of digits is either very large or very small.
\begin{lemma}\label{lem_bound_on_exception}
	For all pairs of positive integers $A, C$ and for all $N$ sufficiently large we have that
	\[
		\#\left\{n<N:\left|s_q(n)-\frac{q-1}{2}\log_q(N)\right|>Jm/C\right\} < \frac{N}{(\log(N))^A}.
	\]
\end{lemma}
\begin{proof}
	Let $T=\lceil \log_q(N) \rceil$, and note that for $N$ sufficiently large we have
	\[
		\left|\frac{q-1}{2}\log_q(N)-\frac{q-1}{2}T\right|<Jm/2C.
	\]
	Therefore, by the triangle inequality, it suffices to show that
	\[
		\#\left\{n<N:\left|s_q(n)-\frac{q-1}{2}T\right|>Jm/2C\right\} < \frac{N}{(\log(N))^A}.
	\]
	We will even prove the stronger statement
	\[
		\#\left\{n<q^T:\left|s_q(n)-\frac{q-1}{2}T\right|>Jm/2C\right\} < \frac{q^T}{q(\log(N))^A} < \frac{N}{(\log(N))^A}.
	\]
	Since for $n<q^T$ each combination of $T$ digits appears exactly once, we can view these digits as a string of $T$ many iid random variables, each taking the values between $0$ to $q-1$ uniformly. We conclude the proof by using Hoeffding's inequality:
	\begin{align*}
		\frac 1{q^T}\left\{n < q^T:\left|s_q(n)-\frac{q-1}{2}T\right|>Jm/2C\right\}
		&\ll \exp\left(-\frac{\lambda(\log\lambda)^{1+2\varepsilon}}{4C^2T(q-1)^2}\right)\\
		&\ll \exp(-C'(log\lambda)^{1+2\varepsilon})\\
		&\ll \log(N)^{-D}.
	\end{align*}
\end{proof}

Next we state the so-called carry propagation lemma. It has been proven and used in slight variations in multiple papers (e.g. \cite{Spiegelhofer2018} and \cite{MauduitRivat2010}). We state the version of Toumi \cite[Lemma 3.3]{Toumi2025}.
\begin{lemma}\label{lem_carries}
	Let $r>0$, $\lambda\geq1$, $N\geq 1$ be integers and $\alpha,\beta\in \RR$ with $\alpha>0$ and $\beta\geq 0$. Let $I\subset[0,\infty)$ be an interval in $\RR$ containing exactly $N$ consecutive integers. Then
	\begin{align*}
		&\#\big\{n\in I: s_q(\lfloor\alpha(n+r)+\beta\rfloor)-s_q(\lfloor\alpha(n)+\beta\rfloor)
		\neq\\
		&\hspace{55pt} s_q^{[0,\lambda)}(\lfloor\alpha(n+r)+\beta\rfloor)-
		s_q^{[0,\lambda)}(\lfloor\alpha(n)+\beta\rfloor)\big\}
		<r\left(\frac{N\alpha}{q^\lambda}+2\right)
	\end{align*}
\end{lemma}

The following three results are vital for proving equidistribution results along a short arithmetic progression modulo $m$ in Section \ref{sec_uniform}. 

\begin{lemma}\label{interated_van_der_corput}
	Let $Q\geq1$ be an integer. Assume that $J$ is a finite non-empty interval in $\ZZ$ and $g:J\rightarrow \{z\in \CC: |z| = 1 \}$. Then we have for all integers $M_1,\dots,M_{Q}\geq1$ and for all $R\geq 1$ that
	\[
	\left|\frac{1}{|J|}\sum_{n\in J}g(n)\right|^{2^Q}\ll \frac {1}{R^Q}\sum_{r\in\{1,\dots,R\}^Q}|K(r_1M_1,\dots,r_{Q}M_{Q})|+O\left(\frac{(M_1+\dots+M_{Q})R}{|J|}+\frac{1}{R}\right),
	\]
	where
	\[
	K(m_1,\dots,m_{Q}):=\frac{1}{|J|}\sum_{n\in J}\prod_{\varepsilon\in\{0,1\}^Q}\mathcal{C}^{|\varepsilon|}g\left(n+\sum_{\ell=1}^{Q}\varepsilon_{\ell}m_{\ell}\right).
	\]
\end{lemma}
This result is obtained by alternately applying the Cauchy--Schwarz inequality and the van der Corput inequality to the left-hand side (cf. Spiegelhofer \cite{Spiegelhofer2023}).

Proposition \ref{prop_equi} treats the distribution of the sum of digits function in different bases $q_i$ along arithmetic progressions modulo $m$. We reduce this question to a sum over $n\in [1,\lcm(q_1^{a_i},\dots, q_t^{a_t})]\cap \NN$ for some positive integers $a_i$, where each sum of digits functions is restricted such that it only depends on the residue class of $n$ modulo $q_i^{a_i}$. If the bases $q_i$ are coprime, then the Chinese Remainder Theorem states that the residue classes are independent and therefore we can separate the contribution of each $q_i$. The next lemma and its corollary states that in the multiplicatively independent setting a certain independence between the contribution of one base $q_i$ and the remaining bases holds.
\begin{lemma} \label{small_gcd}
	Let $q_1,\dots,q_t$ be multiplicatively independent integers. Let $N\in \NN$, further let $a_i\in \ZZ$ be minimal such that $q_i^{a_i}>N$. Then there exists an $\varepsilon>0$ and an index $j$ such that for all $N$ sufficiently large it holds that
	\[
		\gcd(q_j^{a_j}, \lcm(q_1^{a_1},\dots, q_{j-1}^{a_{j-1}}, q_{j+1}^{q_{j+1}}, \dots,q_t^{a_t}))<N^{1-\varepsilon}.
	\]
\end{lemma}
\begin{proof}
	Let $p$ be any prime number that divides $q_1$. Define $e_i=\nu_p(q_i)$ to be the multiplicity of $p$ in $q_i$. Let $j\in \{1,\dots,k\}$ be the integer that maximizes
	\[
		\frac{e_j}{\log q_j}.
	\]
	Note that this maximum is unique since all the $q_i$ are multiplicatively independent. We now see that 
	\[
		\nu_p(q_j^{a_j}) = e_j\cdot a_j  = \left\lceil\frac{e_j}{\log q_j}\log(N)\right\rceil.
	\]
	Furthermore, we have that
	\[
		\nu_p(\lcm(q_1^{a_1},\dots, q_{j-1}^{a_{j-1}}, q_{j+1}^{q_{j+1}}, \dots,q_t^{a_t})) = \max_{i\neq j} \left\lceil\frac{e_i}{\log q_i}\log(N)\right\rceil.
	\]
	Therefore, we see that
	\[
		\gcd(q_j^{a_j}, \lcm(q_1^{a_1},\dots, q_{j-1}^{a_{j-1}}, q_{j+1}^{q_{j+1}}, \dots,q_t^{a_t})) < 
		q_j^{a_j}/p^{\delta \log(N)} = N^{1-\varepsilon},
	\]
	for some $\delta$, where $\delta \log(N)$ satisfies
	\[
		\delta \log(N) = \left\lceil\frac{e_j}{\log q_j}\log(N)\right\rceil-\max_{i\neq j} \left\lceil\frac{e_i}{\log q_i}\log(N)\right\rceil.
	\]
	For sufficiently large $N$ we have that $\delta$ is nonzero because the unique maximum is achieved at $q_j$.
\end{proof}
To state the corollary below, we need some additional notation. Note that any strongly $q$-multiplicative function $f$ satisfies
\[
	f(n)= \prod_{i\geq 0}f(\delta_i(n)).
\]
Similarly to the sum of digits function, for any interval $J\subset\NN$, we will denote by $f^J$ the function
\[
	f^J(n):=  \prod_{i\in J}f(\delta_i(n)).
\]  
\begin{corollary}\label{cor_independence}
	Let $q_1,\dots,q_t$ be multiplicatively independent integers, and let $f_i$ be strongly $q_i$-multiplicative functions, such that $|f_i(n)\leq 1|$ for all $n\in \NN$.\\
	Let $M\in \NN$ be sufficiently large and $q_i^{x_i}>M$, with $x_i$ minimal. Further, let $y<x_1$ be an integer, such that $q_1^{y}>M^{1-\varepsilon}$, where the $\varepsilon$ is as in Lemma \ref{small_gcd} and we assume that upon relabelling that $j=1$ in Lemma \ref{small_gcd}. Then, if $N>M^{t+\varepsilon}$, we have that
	\begin{align*}
	    \frac 1N &\sum_{n\leq N}  f_1^{[y,x_1)}(n) \prod_{i=2}^t f_i^{[0,x_i)}(n) = \\
        &\frac 1{q_1^{x-y}}\sum_{n\leq q_1^{y-x}} f_1^{[y,x)}(n) \frac 1{\lcm(q_2^{x},\dots,q_t^{x})}\sum_{n\leq \lcm(q_2^{x},\dots,q_t^{x})} \prod_{i=2}^t f_i^{[0,x)}(n) + O(N^{-\varepsilon})
	\end{align*}
\end{corollary}
\begin{proof}
	The function $f_1^{[y,x_1)}(n) \prod_{i=2}^t f_i^{[0,x_i)}(n)$ is periodic with period $\lcm(q_1^{x_1},\dots,q_t^{x_t})$. Hence, we may assume that $\lcm(q_1^{x_1},\dots,q_t^{x_t})|N$ by introducing an error term
	\[
		O\left(\frac{\lcm(q_1^{x_1},\dots,q_t^{x_t})}{N}\right) = O(N^{-\varepsilon})
	\]
	Let $w_1<q_1^{x_1-y}$ and $w_2<\lcm(q_2^{x_2},\dots,q_t^{x_t})$ be two positive integers. By Lemma \ref{small_gcd}, we know that we can find a $w_3\leq \gcd(q_1^{x_1}, \lcm(q_2^{x_2},\dots,q_t^{x_t}))$ such that
	\[
		w_1q_1^y+w_3 \equiv w_2 \bmod q_1^{x_1}
	\]
	and
	\[
		w_1q_1^y+w_3 \equiv w_2 \bmod \lcm(q_2^{x_2},\dots,q_t^{x_t}).
	\]
	In other words, the pair $(w_1q_1^y+w_3, w_2)$ corresponds to a unique residue class modulo $\lcm(q_1^{x_1},\dots,q_t^{x_t})$. Note also that $w_3<q_1^y$ by the assumptions on $y$. Therefore, we have that 
	\[
		f_1^{[y,x_1)}(w_1)=f_1^{[y,x_1)}(w_1+w_3).
	\]
	For any pair $(w_1,w_2)$ we can find either
	\[
		\frac{q_1^{x_1-y}\lcm(q_2^{x_2},\dots,q_t^{x_t})}{\lcm(q_1^{x_1},\dots,q_t^{x_t})} \; \text{ or } \; 	\frac{q_1^{x_1-y}\lcm(q_2^{x_2},\dots,q_t^{x_t})}{\lcm(q_1^{x_1},\dots,q_t^{x_t})} + 1
	\]
	many $w_3$. Because $w_1$ and $w_2$ are arbitrary, we can conclude that
	\begin{align*}
	    \frac 1N &\sum_{n\leq N}  f_1^{[y,x_1)}(n) \prod_{i=2}^t f_i^{[0,x_i)}(n) =\\
        &\frac 1{q_1^{x_1-y}}\sum_{n\leq q_1^{x_1-y}} f_1^{[y,x_1)}(n) \frac 1{\lcm(q_2^{x_2},\dots,q_t^{x_t})}\sum_{n\leq \lcm(q_2^{x_2},\dots,q_t^{x_t})} \prod_{i=2}^t f_i^{[0,x_i)}(n) + O(N^{-\varepsilon}),
	\end{align*}
	where we introduce another error term of order
	\[
		O\left(\frac{q_1^{x_1-y}\lcm(q_2^{x_2},\dots,q_t^{x_t})}{N}\right)=O(N^{-\varepsilon})
	\]
	by the implicit assumption that each $w_3$ appears equally often.
\end{proof}
\section{Finding the residue class $L'$} \label{sec_Proof_of_claims}
In this section we prove Proposition \ref{prop_shift}. Let $r= \gcd(a(q_1-1),b(q_2-1))$, we aim to find an $L'<q_1^{\xi_{1}-1}q_2^{\xi_{2}-1}$ such that
\begin{equation}\label{eq_def_f}
	f(n+d_j)-f(n) = as_{q_1}(n+d_j)-bs_{q_2}(n+d_j)-as_{q_1}(n)+bs_{q_2}(n) = jm
\end{equation}
for all $n\equiv L' \bmod \lcm(q_1^{\xi_{1}},q_2^{\xi_{2}})$ and $-J\leq j \leq J$ such that $\gcd(a-b,r)\mid\gcd(j,r)$.

We use Lemma \ref{small_gcd} as a substitute for the Chinese Remainder Theorem. Recall, that it states that if the integers $\xi_1$ and $\xi_2$ are minimal such that $q_1^{\xi_1}>N$ and $q_2^{\xi_2}>N$, then
\[
	\gcd(q_1^{\xi_1}, q_2^{\xi_2})< N^{1-\varepsilon}.
\]
In particular if we have an integer $n_1$ such that $n_1q_1^{\lceil(1-\varepsilon)\xi_1\rceil}<q_1^{\xi_1}$ and an integer $n_2<q_2^{\xi_2}$ then there exists an $n_3<q_1^{(1-\varepsilon)\xi_1}$ such that
\[
	n_1q_1^{\lceil(1-\varepsilon)\xi_1\rceil} + n_3 \equiv n_2 \bmod \gcd(q_1^{\xi_1}, q_2^{\xi_2})
\]
Therefore, we can now apply the Chinese remainder theorem and get that there exists an $n$ such that
\[ 
	n \equiv n_1q_1^{\lceil(1-\varepsilon)\xi_1\rceil}+n_3 \bmod q_1^{\xi_1} \text{ and } n \equiv n_2 \bmod q_2^{\xi_2}
\]

We further observe the identity $s_q(n)\equiv n \bmod q-1$. Let $r = \gcd(a(q_1-1), b(q_2-1))$. Then we have that
\[
as_{q_1}( L_1'+d_j)-as_{q_1}( L_1') - (bs_{q_2}( L_2+d_j)-bs_{q_2}( L_2)) \equiv (a-b)d_j \bmod r.
\]
By assumption, we have that $m\equiv 1 \bmod r$. Hence, we deduce that $\gcd(a-b, r) \mid \gcd (j, r)$ is a necessary condition to have a solution. Let us now show that it is sufficient.\\
Assume that $\gcd(a-b, r) \mid \gcd (j, r)$ is satisfied. Let $0 \leq y < r$ be the smallest number such that $(a-b)y \equiv j\bmod r$, say $y \leq q_1^{\tilde{c}}$.\\

We now define $d_j$. Let $d_j$ be the number whose digits in base $q_1$ are $(00\mathrel{{q_1}{-}{1}})^{(j+J+1)\beta}0^{\lfloor(2J+1)\beta/\varepsilon\rfloor}y$. Let $x= \lfloor(2J+1)\beta/\varepsilon\rfloor$ and $\xi_{1} = (2J+1)\beta + x + 
\tilde{c} + 1$. We can rewrite $d_j = q_1^{x}\sum_{\ell=-J}^{j}zq_1^{(\ell+J+1)\beta}$+y, where $z$ is the positive integer with the digit representation $(00\mathrel{{q}{-}{1}})^{\beta}$.\\

Note that if $n< q_1^{x}$ and $q_1^{\tilde{c}}\mid n$, then we have that
\[
	as_{q_1}(n+d_j)-as_{q_1}(n) = as_{q_1}(d_j),
\]
i.e. the equation above is independent of $n$.\\

Let $c_j$ be such that $q_2^{c_j}\leq d_j < q_2^{c_j+1}$. We aim to find two numbers $L_1' = q_1^{x}L_1<q_1^{(2J+1)\beta+x}$ and $L_2<q_2^{c_J+a(q_1-1)} = q_2^{\xi_2}$. We further want to choose $L_2$ such that $q_1^{\tilde{c}}|L_2$. By Lemma \ref{small_gcd}, we can find an $L_3$ such that $L_1'+L_3$ are compatible with $L_2$. We have that
\[
	L_1'+d_j < q_1^{\xi}.
\]
Therefore, if $n \equiv L_1'\bmod q_1^{\xi_1}$, then
\[
	as_{q_1}(n+d_j)-as_{q_1}(n) = as_{q_1}( L_1'+d_j)-as_{q_1}( L_1').
\]
Similarly, because
\[
	L_2+d_j< q_2^{x_2},
\]
we have that if $n\equiv L_2 \bmod q_2^{\xi_2}$, it follows that
\[
	bs_{q_2}(n+d_j)-bs_{q_2}(n) = bs_{q_2}( L_2+d_j)-bs_{q_2}( L_2).
\]
Combining these observations yields that if $L_1'$ and $L_2$ are such that
\begin{equation}\label{eq_L'}
	as_{q_1}( L_1'+d_j)-as_{q_1}( L_1') - (bs_{q_2}( L_2+d_j)-bs_{q_2}( L_2))=jm,
\end{equation}
we also have that
\[
	f(n)=jm
\]
whenever $n\equiv L_1'+L_3 \bmod q_1^{x_1}$ and $n\equiv L_2 \bmod q_2^{\xi_2}$. Therefore, we are left to find such values of $L_1'$ and $L_2$, satisfying \eqref{eq_L'}.\\

Our first aim is to define $L_2$. We see from that
\[
	a(s_{q_1}(n+d_j)-as_{q_1}(n)) \equiv  ad_j \bmod a(q_1-1)
\]
Hence, we need that
\begin{equation}\label{eq_mod_q2}
	ad_j - (bs_{q_2}(L_2+d_j)-bs_{q_2}(L_2))\equiv j \bmod a(q_1-1).
\end{equation}
If we assume that $L_2$ is of the form $q_2^{c_j+t_j}-q_2^{c_j}+u_j$ for some integer $t_j\geq0$ and $0\leq u_j<q_2^{c_j}$. Then we have that
\[
	bs_{q_2}(n+d_j)-bs_{q_2}(n) = b(t_j(q_2-1) + s_{q_2}(u_j+d_j)-s_{q_2}(u_j)).
\]
Therefore, equation \eqref{eq_mod_q2} can be written as 
\[
	b(t_j(q_2-1) + s_{q_2}(u_j+d_j)-s_{q_2}(u_j)) \equiv ad_j-j \bmod a(q_1-1).
\]
Because we have that $s_{q_2}(u_j+d_j)-s_{q_2}(u_j) = d_j + k(q_1-1)$ for some $k\in \ZZ$ and $(a-b)d_j \equiv j \bmod r$, we have that the above equation indeed has a solution for $t_j$ and we may assume that $0\leq t_j < a(q_1-1)$.\\

Since we do not know the digit representation of $d_j$ in base $q_2$, we want to apply Lemma \ref{Addition}. We will use the identity
\[
	s_{q_2}(n) = s_{q_2}^{[c_J+aq_1, \infty)}(n) + \sum_{j=-J}^{J}s_{q_2}^{[c_{j-1}+aq_1, c_j)}(n) + s_{q_2}^{[c_{j}, c_j+aq_1)}(n)
\]

Here we use that $c_{-J-1}=0$. Since $L_2+d_j< q_2^{c_J+aq_1}$, we have that $s_{q_2}^{[c_J+aq_1, \infty)}(L_2+d_j)=0$. It remains to bound the contribution of the other summands.\\

First we note that $s_{q_2}^{[c_{j}, c_j+aq_1)}(n)$ consists only of $aq_1$ many digits, hence we may bound this by $aq_1q_2$. This holds for all $-J\leq j\leq J$, their combined contribution to the sum is therefore bounded by $(2J+1)aq_1q_2$.\\

We will apply Lemma \ref{Addition} to each of the other summands. Because there might be a carry from the lower digits into this interval, we need to apply Lemma \ref{Addition} not only for the shifts $d_j$, but we also need to take into account these carries. Therefore, the set of shifts will be
\[
	U=\{d_j: -J\leq j\leq J\} \cup \{d_j+q_2^{c_k+q_1-1}:-J\leq j,k\leq J\}
\]
By Lemma \ref{Addition}, we know that there exists a $k_{j}$ such that $k_jq_1^{c_{j-1}+q_1}<q_1^{c_j}$ and
\[
	s_{q_1}^{[c_{j-1}+q_1, c_j)}(k_{j}q_1^{c_{j-1}+q_1}+u)-	s_{q_1}^{[c_{j-1}+q_1, c_j)}(k_{j}q_1^{c_{j-1}+q_1})<\log(N)
\]
for all $-J\leq j \leq J$ and for all $u\in U$. Using the observation above, we choose $t_{j}<a(q_1-1)$ such that
\[
	s_{q_1}\left(\sum_{\ell=-J}^{j}\left(k_{\ell}q^{c_{\ell-1}+q_1}+ q_2^{t_{\ell}+c_{\ell}}-q_2^{c_{\ell}}\right)+d_j\right)-	s_{q_1}\left(\sum_{\ell=-J}^{j}\left(k_{\ell}q^{c_{\ell-1}+q_1}+ q_2^{t_{\ell}+c_{\ell}}-q_2^{c_{\ell}}\right)\right)\equiv 0 \bmod a(q_1-1).
\]
We define
\[
	L_2' = \sum_{j=-J}^{J}k_jq^{c_{j-1}+q_1-1}+ q_2^{t_j+c_j}-q_2^{c_j}.
\]
We conclude that
\[
	bs_{q_2}(L_2'+d_j)-bs_{q_2}(L_2') < b(2J+1)\log(N) + (2J+1)abq_1q_2
\]
and that
\[
	b\left(s_{q_1}\left(L_2'+d_j\right)-s_{q_1}\left(L_2'\right)\right)\equiv ad_j-j \bmod a(q_1-1).
\]
The last step of defining $L_2$ is that we need to ensure that $q_1^{\tilde{c}}\mid L_2$. If $\gcd(q_1,q_2)=1$, we find $L_2$ by the Chinese Remainder Theorem and the congruence condition that
\[
	L_2 \equiv L_2' \bmod q_2^{c_j+a(q_1-1)} \text{ and } L_2 \equiv 0 \bmod q_1^{\tilde{c}}.
\]
If $\gcd(q_1,q_2)\neq 1$, we write $q_1 = q_{1,1}q_{1,2}$ and assume that $q_{1,1}$ is maximal such that $\gcd(q_{1,1},q_2)=1$ and $\gcd(q_{1,2},q_2) \neq 1$. It follows that 
\[
	q_{1,2}^{\tilde{c}}|q_2^{\tilde{c}'}.
\] 
We deduce that this condition only influences the digits at positions $0$ to $\tilde{c}'<c_{-J}$. In the application of Lemma \ref{Addition} to the lowest interval we can therefore impose the condition that $k_{-J} \equiv 0 \bmod q_{1,2}^{\tilde{c}}$ and deduce that
\[
	L_2' \equiv 0 \bmod q_{1,2}^{\tilde{c}}.
\]
Now we again apply the Chinese Remainder Theorem to the congruence condition
\[
	L_2 \equiv L_2' \bmod q_2^{c_j+a(q_1-1)} \text{ and } L_2 \equiv 0 \bmod q_{1,1}^{\tilde{c}}
\]
and find that $q_1^{\tilde{c}}\mid L_2$.\\

Next, we choose $L_1$. We follow the approach by Spiegelhofer \cite{Spiegelhofer2023a}. We use the digit representation of $d_j$ and observe the following three identities.
\[
\begin{array}{cccc}
	&0&0&0\\
	+&0&0&\mathrel{{q_1}{-}{1}}\\
	\cline{2-4}
	&0&0&\mathrel{{q_1}{-}{1}}
\end{array}
\]
\[
\begin{array}{cccc}
	&0&0&1\\
	+&0&0&\mathrel{{q_1}{-}{1}}\\
	\cline{2-4}
	&0&1&0
\end{array}
\]
\[
\begin{array}{cccc}
	&0&\mathrel{{q_1}{-}{1}}&1\\
	+&0&0&\mathrel{{q_1}{-}{1}}\\
	\cline{2-4}
	&1&0&0
\end{array}
\]
From this we see that in each block of $3$ digits, we can change the value of the sum of digits function by $0$ or by $\pm (\mathrel{{q_1}{-}{1}})$ without any carries into another block.\\
Without loss of generality we can assume that $\gcd(a-b,r)|\gcd(J,r)$ as otherwise we can replace $J$ by $J-1$. For $d_{-J}$, we can find a positive integer $x_{-J}<q_1^{3\beta}$ such that
\[
	a(s_{q_1}(x_{-J}q_1^{x}+d_{-J})-s_{q_1}(x_{-J}q_1^{x}))=a(q_1-1)n
\]
for any $-3\beta\leq n \leq 3\beta$. However, we notice that
\[
	|jm + bs_{q_2}(L_2+d_j)-bs_{q_2}(L_2)| \ll \log(N)^{0.51}\ll \beta
\]
and
\[
	a(q_1-1)\mid jm + bs_{q_2}(L_2+d_j)-bs_{q_2}(L_2) - ad_j.
\]
Therefore, we can find an $x_{-J}$ such that
\[
	a(s_{q_1}(x_{-J}q_1^{x}+d_{-J})-s_{q_1}(x_{-J}q_1^{bla})) = jm + bs_{q_2}(L_2+d_j)-bs_{q_2}(L_2).
\]
Now we proceed by induction on $j$. We assume that we have found $x_{k}<q_1^{3\beta}$, for $-J\leq k\leq j-1$ such that
\[
	a\left(s_{q_1}\left(\sum_{\ell = -J}^{j-1}x_{j}q_1^{x+(\ell+J+1)\beta}+d_{k}\right)-s_{q_1}\left(\sum_{\ell = -J}^{j-1}x_{j}q_1^{x+(\ell+J+1)\beta}\right)\right) - b (s_{q_2}(L_2+d_k) + bs_{q_2}(L_2)) = km
\]
for all $-J\leq k\leq j-1$. Since there are no carries between any of block of $3$ digits and by the structure of $d_j$, we have that
\begin{align*}
	&s_{q_1}(x_j+z)-s_{q_1}(x_j) + s_{q_1}\left(\sum_{\ell = -J}^{j-1}x_{j}q_1^{x+(\ell+J+1)\beta}+d_{j-1}\right)-s_{q_1}\left(\sum_{\ell = -J}^{j-1}x_{j}q_1^{x+(\ell+J+1)\beta}\right) =\\
	&\hspace{20pt}
	s_{q_1}\left(\sum_{\ell = -J}^{j}x_{j}q_1^{x+(\ell+J+1)\beta}+d_{j}\right)-s_{q_1}\left(\sum_{\ell = -J}^{j-1}x_{j}q_1^{x+(\ell+J+1)\beta}\right).
\end{align*}
If $\gcd(a-b,r)\nmid\gcd(j,r)$ we define $x_j$ such that
\[
	a(s_{q_1}(x_j+z)-s_{q_1}(x_j))=0
\]
Otherwise, we define $x_j$ such that
\begin{equation}\label{eq_def_x1}
	a(s_{q_1}(x_j+z)-s_{q_1}(x_j))=jm - \delta_j,
\end{equation}
where
\[
	\delta_j =  - bs_{q_2}(L_2+d_j)+bs_{q_2}(L_2) +a\left(s_{q_1}\left(\sum_{\ell = -J}^{j-1}x_{j}q_1^{x+(\ell+J+1)\beta}+d_{j-1}\right)-s_{q_1}\left(\sum_{\ell = -J}^{j-1}x_{j}q_1^{x+(\ell+J+1)\beta}\right)\right).
\]
Note that
\begin{align*}
	&a\left(s_{q_1}\left(\sum_{\ell = -J}^{j-1}x_{j}q_1^{x+(\ell+J+1)\beta}+d_{j-1}\right)-s_{q_1}\left(\sum_{\ell = -J}^{j-1}x_{j}q_1^{x+(\ell+J+1)\beta}\right)\right) = \\
	& \hspace{20pt} (j-1)m + bs_{q_2}(L_2+d_j)- bs_{q_2}(L_2) \leq \\
	& \hspace{20pt} (j-1)m + b(2J+1)\log(N) + (2J+1)abq_1q_2
\end{align*}
Hence we have that
\[
	|jm - \delta_j| \leq |(2j-1)m + 3(b(2J+1)\log(N) + (2J+1)abq_1q_2) \ll \log(N)^{0.51} \leq \beta
\]
Therefore, we conclude that there exists an $x_j<q_1^{3\beta}$ such that
\begin{equation}\label{eq_def_x}
	a(s_{q_1}(x_j+z)-s_{q_1}(x_j))=jm - \delta_j
\end{equation}
exists. We define $L_1=\sum_{\ell = -J}^{J}x_{j}q_1^{x+(\ell+J+1)\beta}$ and conclude that
\[
	a(s_{q_1}(L_1+d_{k})-s_{q_1}(L_1)) - b (s_{q_2}(L_2+d_k) + bs_{q_2}(L_2)) = km
\]
for all $-J\leq k\leq J$. By the observation at the beginning of this section we see that there exists an $L_3<q_1^{\lfloor(2J+1)\beta/\varepsilon\rfloor}$ such that
\[
	L_1q_1^{\lfloor(2J+1)\beta/\varepsilon\rfloor} + L_3 \equiv L_2 \bmod \gcd(q_1^{x_1},q_2^{x_2}).
\] 
Let $L'$ be the residue class modulo $\lcm (q_1^{x_1},q_2^{x_2})$ such that
\[
	L' \equiv L_1q_1^{\lfloor(2J+1)\beta/\varepsilon\rfloor} + L_3 \bmod q_1^{\xi_1} \text{ and } L' \equiv L_2 \bmod q_2^{\xi_2}.
\]
We conclude that if $n\equiv L' \bmod \lcm (q_1^{x_1},q_2^{x_2})$, then we have that
\[
	f(n+d_j)-f(n)=jm
\]
for all $-J\leq j\leq J$ such that $\gcd(a-b,r) \mid \gcd(j,r)$.
\section{Study of arithmetic Progression}\label{sec_Arith_Prog}
In this section we prove Proposition \ref{thm_concentration}, i.e. that for all $N<k<2N$ except at most $O\left(N\log(N)^{-A}\right)$ the following inequality holds.
\[
	\left|s_q(a+dk)-\frac{q-1}2\log_{q}\left(N\right) - \frac{q-1}2\log_{q}(d')-s_q^{(e)}(a)\right|\leq Jm
\]
where $d=q^ed'$.
\begin{proof}
	First we notice that $s_q(a+dk)=s_q^{[0,e)}(a)+s_q(a'+d'k)$, where $a'= \left\lfloor a/q^{e}\right\rfloor$. Therefore, it suffices to show that for almost all $k$ we have
	\[
		\left|s_q(a'+d'k)-\frac{q-1}2\log_{q}\left(N\right)\right|\leq Jm
	\]
	Let $\kappa=\min\{m:q^m\geq d'\}$. Then we have for all integers $k\geq 0$ that
	\[
		s_q(a'+d'k)= s_q^{[0,(2/3)\log(N))}((a'+d'k)) + s_q^{[(2/3)\log(N),\kappa)}((a'+d'k)) + s_q\left(\left\lfloor\frac {d'k}{q^{\kappa}}+\sigma\right\rfloor\right)
	\]
	where $\sigma=ap^{-\kappa}<1$. We will see that the admissibility assumption is only needed to treat the middle summand. Hence, in the case where $d'<d<N^{1/2}$ we can replace $(2/3)$ by any rational number $x<1$ such that  $x\log(N)=\kappa$ if $N$ is sufficiently large and hence the middle summand disappears.\\
	We now analyse the exceptional set of the left summand, the middle summand and the right summand separately.\\
	For the left summand, we see that for each $\omega\in \{0,\dots, N^{2/3}-1\}$ the equation
	\[
		(a+d'k)^{[0,(2/3)\log(N))}=\omega
	\]
	has at most $cN^{1/3}$ solutions since $d'$ is coprime to $q$. Additionally, we see that by Lemma \ref{lem_bound_on_exception}
	\[
		\#\left\{n<N^{2/3}: \left|s_q(n)-\frac{(q-1)}3 \log_q N\right|>JM/C\right\}=O\left(\frac{N^{2/3}}{\log(N)^A}\right).
	\]
	Combining these two estimates, we get that
	\[
		\#\left\{n<N: \left|s_q^{[0,(2/3)\log(N))}((a'+d'k))-\frac{(q-1)}3 \log_q N\right|>JM/C\right\}=O\left(\frac{N}{\log(N)^A}\right).
	\]
	For the middle summand, we aim to apply Lemma \ref{odd-elimination}. We define a sequence for integers $a_0, \dots, a_x$, for some $x\in \NN$ such that $2/3\log(N)=a_0<a_1<\dots<a_x=\kappa$, and $\eta\log(N)/2 \leq a_{i+1}-a_i \leq \eta\log(N)$. Such a sequence can be easily found by taking $a_{i+1}-a_{i}=\eta\log(N)$ for all $i$, until $\kappa - a_{i} < 2\eta\log(N)$. Now we set $a_{i+1}=a_x=\frac{a_i+\kappa}{2}$. This sequence satisfies the assumptions above.
	Hence, we get that
	\[
		 s_q^{[2/3\log(N),\kappa)}((a'+d'k)) =  \sum_{i=0}^{x-1} s_q^{[a_i,a_{i+1})}((a'+d'k))
	\]
	Because $\eta<1/5$, we can use Lemma \ref{odd-elimination} and we see that for each $\omega\in \{0,\dots, q^{a_{i+1}-a_i}\}$ the equation
	\[
		(a+d'k)^{[a_i,a_{i+1})}=\omega
	\]
	has $cq^{\log_q(n)-a_{i+1}+a_{i}}$ many solutions. By Lemma \ref{lem_bound_on_exception} we see that there are at most $O\left(\frac{q^{a_{i+1}-a_{i}}}{\log(N)^A}\right)$ many $\omega$ such that 
	\[
		\#\left\{n<q^{a_{i+1}-a_{i}}: \left|s_q(n)-\frac{(q-1)}2 (a_{i+1}-a_{i})\right|>JM/C\right\}
	\]
	We have that for each $0\leq i<x$, hence we have by the triangle inequality that
	\[
		\#\left\{n<N: \left|s_q^{[(2/3)\log(N),\kappa)}((a'+d'k))-\frac{(q-1)}2 (\kappa-\frac{2\log(N)}3)\right|>xJM/C\right\}=O\left(\frac{N}{\log(N)^A}\right).
	\]
	We now analyse the right summand:
	\[
		s_q\left(\left\lfloor\frac {d'k}{q^{\kappa}}+\sigma\right\rfloor\right)
	\]
	and notice that the equation
	\[
		\left\lfloor\frac{d'k}{q^{\kappa}}+\sigma\right\rfloor=\omega
	\]
	has at most $q$ solutions for each $\omega\in \{0,\dots, N\} $. Hence, we again use Lemma \ref{lem_bound_on_exception} and get similarly as above that
	\[
		\#\left\{n<N: \left|s_q\left(\left\lfloor\frac {d'k}{q^{\kappa}}+\sigma\right\rfloor\right)-\frac{(q-1)}2 \log(N)\right| >JM/C\right\} = O\left(\frac{N}{\log(N)^A}\right).
	\]
	We now pick $C=x+2$ and see that if $N$ is not in one of the exceptional sets of one of the summands we have that
	\[
		\left|s_q(a+dk)-\frac{q-1}2\log_{q}\left(N\right) - \frac{q-1}2\log_{q}(d')-s_q^{(e)}(a)\right|\leq Jm
	\]
	We also see by the union bound that there are at most $O\left(\frac{N}{(\log(N))^A}\right)$ many $N$ such that the above equation fails hence proving the proposition.
\end{proof}
\section{Equidistribution modulo $m$ along Arithmetic Progressions}\label{sec_uniform}
In this section we prove Proposition \ref{prop_equi}. We want to study the order of magnitude of 
\[
	\mathcal{A}:=\#\left\{ n\in[N,2N]: s_{q_i}(d_in+a_i)\equiv \ell_i\bmod m_i \text{ for all } 1\leq i \leq t
	\right\}.
\]
We reformulate equality \eqref{eq_prop_size} by summing over the indicator set and by specifying the $o(1)$ term. We obtain
\begin{equation}\label{equal_distr}
	\left|\sum_{\substack{N<n\leq 2N\\n\in \mathcal{A}}}1-\frac{N}{\prod_{i=1}^{t}m_i} \right| \ll \frac{N}{(\log(N))^A}.
\end{equation}

Now, we use the orthogonality of characters, a standard method to detect the congruence condition:
\[
\frac{1}{m_i}\sum_{h_i=0}^{m_i-1}e\left(\frac{h_i}{m_i}(\ell_i-s_{q_i}(n))\right)=\begin{cases}
	1, & \text{if } s_{q_i}(n)\equiv \ell_i \bmod m_i,\\
	0, & \text{otherwise}
\end{cases}
\]
Plugging this into equation \eqref{equal_distr}, we get that
\begin{equation}
	\left|\sum_{N<n\leq 2N}\frac{1}{\prod_{i=1}^{t}m_i}\sum_{h_1=0}^{m_1-1}\cdots \sum_{h_t=0}^{m_t-1}e\left(
	\sum_{i=1}^{t}\frac{h_i}{m_i}(\ell_i-s_{q_i}(n))\right)-\frac{N}{\prod_{i=1}^{t}m_i} \right| \ll \frac{N}{(\log(N))^A}.
\end{equation}
We switch the order of summation and consider two cases depending on the values of $h_i$. If the tuple $(h_1,\dots, h_r)= (0,\dots,0)$ we get that
\begin{equation}
	\left|\frac{N}{\prod_{i=1}^{t}m_i} - \frac{N}{\prod_{i=1}^{t}m_i}\right| = 0.
\end{equation}
Therefore, it is enough to prove that for any tuple $(h_1,\dots, h_r)\neq(0,\dots,0)$ the following equation holds:
\begin{equation}
	\left|\sum_{N<n\leq 2N}e\left(\sum_{i=1}^{t}\frac{h_i}{m_i}s_{q_i}(nd_i+a)\right)\right| \ll \frac{N}{(\log(N))^A}.
\end{equation}
We prove even a stronger statement:
\begin{proposition}\label{prop_theta_equi}
	Let $H$ be a positive constant. Let $q_1, \dots, q_t$ be a collection of multiplicatively independent bases. Let $d_i, a_i$ be a collection of positive integers such that, for a sufficiently small $\eta$, $d_i$ is $N^{\eta}$-admissible in base $q_i$ and $a_i<d_i<N^H$. Let $\theta_i$ be a collection of real numbers, not all zero. Then there exists a $j\in \{1,2,\dots, t\}$ such that $\theta_j\neq 0$ and the following holds:
	\begin{equation}
		\left|\sum_{N<n\leq 2N}e\left(\sum_{i=1}^{t}\theta_is_{q_i}(nd_i+a_i)\right)\right| \ll N \exp\left(-c||(q_j-1)\theta_j||^2\log N\right)
	\end{equation}
\end{proposition}
We see that using $m_i\leq c (\log N)^{1/2}/(\log\log N)^{1/2+\varepsilon}$ and that $\gcd(q_i-1,m_i)=1$, we see that we pick 
\[
	\theta_i = \frac{h_i}{m_i}.
\]
We recall that not all $h_i$ are $0$. If $h_i\neq 0$ we observe that $||(q_i-1)\theta_i||\notin \ZZ$. Therefore, we get an arbitrary logarithmic saving from the proposition above. Hence, we have proven Proposition \ref{prop_equi} under the assumption of Proposition \ref{prop_theta_equi}.

Before we prove Proposition \ref{prop_theta_equi}, we provide a sketch of the proof to highlight the main steps and ideas in the proof.

\noindent\textbf{Sketch of the Proof of Proposition \ref{prop_theta_equi}}\\
Let
\[
    S(\theta_1, \dots, \theta_t) := \frac 1N \sum_{N<n\leq 2N}e\left(\sum_{i=1}^{t}\theta_is_{q_i}(nd_i+a_i)\right).
\]
First we apply Lemma \ref{interated_van_der_corput} to get 
\begin{equation*}
	|S(\theta_1, \dots, \theta_t)|^{2^Q} \ll \frac {1}{R^Q}\sum_{r\in\{1,\dots,R\}^Q}|K(r_1M_1,\dots,r_{Q}M_{Q})| + O\left(\frac{(M_1+\dots + M_{Q})R}{N}+\frac 1R\right),
\end{equation*}
where
\begin{equation*}
	K(r_1M_1,\dots, r_QM_Q) := \frac 1N \sum_{N<n\leq 2N} \sum_{\varepsilon\in \{0,1\}^Q}\mathcal{C}^{|\varepsilon|}
	\e\left(\sum_{i=1}^{t}\theta_i s_{q_i}\left(\left(n+\sum_{\ell = 1}^{Q}\varepsilon_{\ell}r_{\ell}M_{\ell}\right)d_i+a_i\right)\right).
\end{equation*}

Then we choose the $M_{\ell}$ carefully such that we can show that most digits do not contribute to the sum, i.e. such that we are left with
\begin{align*}
	K(r_1M_1,\dots, r_QM_Q) = \frac 1N \sum_{N<n\leq 2N} \sum_{\varepsilon\in \{0,1\}^Q}\mathcal{C}^{|\varepsilon|}
	\e\Bigg(&\theta_1 s_{q_1}^{[a, a+a_1)}\left(\left(n+\sum_{\ell = 1}^{Q} \varepsilon_{\ell}r_{\ell}M_{\ell}\right)d_1 			+a_1\right)\\
	&+\sum_{i=2}^{t}\theta_i s_{q_i}^{[0, a_i)} \left(\left(n+\sum_{\ell = 1}^{Q} \varepsilon_{\ell}r_{\ell}M_{\ell}\right)d_i+a_i\right)\Bigg)\\
	+ O\left(R^Q\left(\frac{1}{q_1^{\rho'_1-a-\rho}}\right)\right).&
\end{align*}
By Corollary \ref{cor_independence}, we can separate the sum above into a sum only involving the base $q_1$ and one which involves all the remaining bases to get
	\begin{align*}
	\hspace{-10pt}K(r_1M_1,\dots, r_QM_{Q}) &= \frac 1{q_1^{a_1}} \sum_{1\leq n \leq q_1^{a_1}} \sum_{\varepsilon\in \{0,1\}^Q}\mathcal{C}^{|\varepsilon|}
	\e\left(\theta_1 s_{q_1}^{[a,a+a_1)}\left(\left(nq_1^{a}+\sum_{\ell=1}^{Q}\varepsilon_{\ell}r_{\ell}M_{\ell}\right)d_1 +a_1\right)\right)\\ 
	&\hspace{-55pt}\times \frac 1{\lcm(q_2^{a_2},\dots,q_t^{a_t})} \sum_{n=0}^{\lcm(q_2^{a_2},\dots,q_t^{a_t})} \sum_{\varepsilon\in \{0,1\}^Q}\mathcal{C}^{|\varepsilon|}
	\e\left(\sum_{i=2}^{t}\theta_i s_{q_i}^{[0,a_i)}\left(\left(n+\sum_{\ell=1}^{Q}\varepsilon_{\ell}r_{\ell}M_{\ell}\right)d_i +a_i\right)\right).
\end{align*}
In the proof, we estimate the part after $\times$ trivially. Afterwards we manipulate the sum over $q_1$ in such a way that, when we put the estimate of $K(r_1M_1,\dots,r_QM_Q)$ back into $S(\theta_1,\dots,\theta_t)$, we get
\begin{align*}
	|S(\theta_1, \dots, \theta_t)|^{2^{Q+1}} &\ll \frac {1}{q_1^{a_1(Q+2)}} \sum_{1\leq n \leq q_1^{a_1}}  \sum_{r\in\{1,\dots,q_1^{a_1}\}^{Q+1}} \sum_{\varepsilon\in \{0,1\}^{Q+1}}\mathcal{C}^{|\varepsilon|}
	\e\left(\theta_1 s_{q_1}^{[0,a_1)}\left(n+\sum_{\ell=1}^{Q+1}\varepsilon_{\ell}r_{\ell}\right)\right),
\end{align*}
which is a Gowers norm. By the author's paper \cite{Jelinek2025}, it is known that this sum is at most
\[
	N^{c||(q_1-1)\theta||^2}
\]
hence proving Proposition \ref{prop_theta_equi}.

Now we rigorously prove the proposition.

\begin{proof}[Proof of Proposition \ref{prop_theta_equi}]
	Without loss of generality we can assume that each $\theta_i$ is nonzero since otherwise we can reduce the summation. We can further assume that $\gcd(d_i,q_i)=1$, because if $d_i=q_i^ed_i'$, then
	\[
		s_{q_i}(nd_i+a_i) = s_{q_i}(nd_i'q_i^e+a_i) = s_{q_i}(nd_i'+\tilde{a_i}) + s_{q_i}^{[0,e)}(a_i).
	\]
	This gives that
	\begin{align*}
		 & \hspace{-40pt}\left|\sum_{N<n\leq 2N}\e\left(\sum_{i=1}^{t}\theta_is_{q_i}(nd_i+a_i)\right)\right| \\
		=&\left|\sum_{N<n\leq 2N} \e\left(\sum_{i=1}^{t}\theta_is_{q_i}(nd_i'+\tilde{a_i})\right)e\left(\sum_{i=1}^{t}\theta_is_{q_i}^{[0,e)}(a_i)\right)\right| \\
		=&\left|e\left(\sum_{i=1}^{t}\theta_is_{q_i}^{[0,e)}(a_i)\right)\right|\cdot \left|\sum_{N<n\leq 2N}\e\left(\sum_{i=1}^{t}\theta_is_{q_i}(nd_i'+\tilde{a_i})\right)\right| \\
		=&\left|\sum_{N<n\leq 2N}\e\left(\sum_{i=1}^{t}\theta_is_{q_i}(nd_i'+\tilde{a_i})\right)\right|.
	\end{align*}
	Therefore, from now on we assume that $\gcd(d_i,q_i)=1$. Let 
	\begin{equation}
		S(\theta_1, \dots, \theta_t) = \frac 1N \left|\sum_{N<n\leq 2N}\e\left(\sum_{i=1}^{t}\theta_is_{q_i}(nd_i+a_i)\right)\right|.
	\end{equation}
	We need to show that 
	\begin{equation}
		|S(\theta_1, \dots, \theta_t)| \ll N^{-c||(q_1-1)\theta_0||^2}.
	\end{equation}

	Let $a=\left\lfloor \eta\log_{q_1}(N)/6 \right\rfloor$ throughout this proof, and let $R=q_1^{\rho}$. Let  the $z_i$ are minimal such that $q_i^{z_i}>q_1^{a+\rho}$. We choose $\rho<a$ maximally such that it satisfies
	\[
		\gcd\left(q_1^{a+\rho},\lcm\left(q_2^{z_2},\dots, q_t^{z_t}\right)\right)\leq q_1^{a}.
	\]
	Such a $\rho$ exists by Lemma \ref{small_gcd} if $N$ is sufficiently large after maybe relabelling the $q_i$. We now apply Lemma \ref{interated_van_der_corput} with $R=q_1^{\rho}$ and some $Q$, which will be specified later. This yields
	\begin{equation}\label{eq_spower}
		|S(\theta_1, \dots, \theta_t)|^{2^Q} \ll \frac {1}{R^Q}\sum_{r\in\{1,\dots,R\}^Q}|K(r_1M_1,\dots,r_{Q}M_{Q})| + O\left(\frac{(M_1+\dots + M_{Q})R}{N}+\frac 1R\right),
	\end{equation}
	where
	\begin{equation}
		K(r_1M_1,\dots, r_QM_Q) := \frac 1N \sum_{N<n\leq 2N} \sum_{\varepsilon\in \{0,1\}^Q}\mathcal{C}^{|\varepsilon|}
		\e\left(\sum_{i=1}^{t}\theta_i s_{q_i}\left(\left(n+\sum_{\ell = 1}^{Q}\varepsilon_{\ell}r_{\ell}M_{\ell}\right)d_i+a_i\right)\right).
	\end{equation}
	
	The aim is to replace each $s_{q_i}(n)$ by $s_{q_i}^{J}(n)$, for some short interval $J$. To achieve this, we choose the $M_{\ell}$ one after the other, such that in each base $q_i$ and for every other small enough interval $I$, $I\cap J =\emptyset$, we can find an index $k$ such that
	\[
		s_{q_i}^{I}\left(\left(n+\sum_{\substack{\ell = 1\\\ell \neq k}}^{Q}\varepsilon_{\ell}r_{\ell}M_{\ell}\right)d_i+a_i\right) = s_{q_i}^{I}\left(\left(n+\sum_{\substack{\ell = 1\\\ell \neq k}}^{Q}\varepsilon_{\ell}r_{\ell}M_{\ell}\right)d_i+a_i+r_{k}M_{k}d_i\right)
	\]
	for almost all $\epsilon_{\ell} \in \{0,1\}\setminus\{k\}$, $N<n\leq 2N $, $r_{\ell}\in\{1,2,\dots, R\}$ and $r_k\in\{1,2,\dots, R\}$. As a consequence of this, we can deduce that
	\begin{equation}\label{eq_observation}
		K(r_1M_1,\dots, r_QM_Q) = \frac 1N \sum_{N<n\leq 2N} \sum_{\varepsilon\in \{0,1\}^Q}\mathcal{C}^{|\varepsilon|}
		\e\left(\sum_{i=1}^{t}\theta_i s_{q_i}^{[0,\infty)\setminus I}\left(\left(n+\sum_{\ell = 1}^{Q}\varepsilon_{\ell}r_{\ell}M_{\ell}\right)d_i+a_i\right)\right).
	\end{equation}
	This is, because $|(\epsilon_1,\dots, \epsilon_{k-1},0,\epsilon_{k+1}, \epsilon_Q)|\not\equiv|(\epsilon_1,\dots, \epsilon_{k-1},1,\epsilon_{k+1}, \epsilon_Q)|\bmod2$, and hence the value of
	\[
		s_{q_i}^{I}\left(\left(n+\sum_{\substack{\ell = 1\\\ell \neq k}}^{Q}\varepsilon_{\ell}r_{\ell}M_{\ell}\right)d_i+a_i\right)
	\]
	appears in the exponent, once with a positive sign and once with a negative sign. Therefore, we conclude that this interval of digits does not contribute to the value of the exponent at all.\\
	
	We will now find such intervals $I$ and integers $M_{\ell}$. An important restriction on the $M_{\ell}$ is that in base $q_1$, $\delta_a\left(M_{\ell}d_1\right)$, the digit at position $a$, is not $0$. We recall that $a=\left\lfloor \eta\log_{q_1}(N)/6 \right\rfloor$ as above.\\
	
	First let $M_1=q_1^{a}$. We now use Lemma \ref{odd-elimination} to choose the values of the other $M_{\ell}$. Let $t=q_1^{a}$ and let for $k\geq 3$\\
	\begin{center}
		\begin{tabular}{lr}
			\vspace{3pt}$x_2=\left\lfloor\log_{q_1}(d_1)\right\rfloor + 2a$ & \hspace{10pt} $x_k=x_{k-1}-a$\\
			$y_2=\left\lfloor\log_{q_1}(d_1)\right\rfloor$ & \hspace{10pt} $y_k=y_{k-1}-a$
		\end{tabular}
	\end{center}
	Then for every $k\geq 2$ such that $y_k\geq 6a$, we can apply Lemma \ref{odd-elimination} to find an $M_k\leq q_1^{24a}$ such that $M_k\equiv 0 \bmod q_1^{a}$ and
	\[
		(M_kd_1)^{[y_k,x_k)} = 1.
	\]
	Let $k_1'$ be the highest index used. Further, since $\gcd(q_1,d_1)=1$, we can find $M_{k_1'+1}\leq q_1^{7a}$ such that $M_{k_1'+1}\equiv0\bmod q_1^a$ and 
	\[
		M_{k_1'+1}d_1=q_1^a \bmod q_1^{7a}.
	\]
	Hence, we choose $k_1=k_1'+1$. We repeat this procedure for all the other bases. In base $q_{i}$ we have the following construction:\\
	Let $t=q_1^{a}$, $f_i = \left\lfloor\frac{\log(q_1)}{\log(q_i)}\right\rfloor$ and let for $k\geq k_{i-1}+2$\\
	\begin{center}
		\begin{tabular}{lc}
			\vspace{3pt}$x_{k_{i-1}+1}=\left\lfloor\log_{q_i}(d_i)\right\rfloor + 2af_i$ & $x_k=x_{k-1}-af_i$\\
			$y_{k_{i-1}+1}=\left\lfloor\log_{q_i}(d_i)\right\rfloor$ & $y_k=y_{k-1}-af_i$
		\end{tabular}
	\end{center}
	Then for every $k\geq k_{i-1}+1$ such that $y_k\geq 6af_i$, we can apply Lemma \ref{odd-elimination} to find an $M_k\leq q_1^{28a}$ such that $M_k\equiv 0 \bmod q_1^{a}$ and
	\[
		(M_kd_1)^{[y_k,x_k)} = 1.
	\]
	Let $k_i'$ denote again the highest index we have used so far. Because $\gcd(q_i,d_i)=1$, we can find 
	\[
		M_{k_i'+1}\leq q_i^{7af_i} \leq q_1^{14a}
	\]
	such that $M_{k_i'+1}\equiv0\bmod q_1^a$ and 
	\[
	M_{k_i'+1}d_i=\gcd\left(q_1^a,q_i^{7af_i}\right) \bmod q_i^{7af_i}.
	\]
	We now define $k_i$ to be $k_i=k_i'+1$.\\
	
	We have now chosen all the values of the $M_{\ell}$'s. In the next step, we pick intervals $I$ and compare the values of 
	\[
		\left(\left(n+\sum_{\ell = 1}^{Q}\varepsilon_{\ell}r_{\ell}M_{\ell}\right)d_i+a_i\right)^{I}
	\]
	for different values of the $\epsilon_{\ell}$ and $r_{\ell}$ in the different bases $q_i$. We distinguish four classes of intervals $I$.\\
	
	First, let $I=[\lfloor\log_{q_i}(d_i)\rfloor+2af_i,\infty)$ We want to find the number of combinations of $N< n\leq 2N$, $r_1$, $r_{\ell}\in\{1,2,\dots,R\}$ and $\epsilon_{\ell}\in\{0,1\}$ for $\ell \in \{2,\dots, Q\}$ such that in base $q_i$
	\[
		\left(\left(n+\sum_{\ell = 2}^{Q}\varepsilon_{\ell}r_{\ell}M_{\ell}\right)d_i+a_i\right)^{I}
			\neq
		\left(\left(n+\sum_{\ell = 2}^{Q}\varepsilon_{\ell}r_{\ell}M_{\ell}\right)d_i+a_i+r_1M_1d_i\right)^{I}.
	\]
	Note that $M_1d_i=q_1^ad_i $, hence, we see by the carry lemma (Lemma \ref{lem_carries}) that there are at most
	\[
		2^QR^Q\left(\frac{N}{q^a}+2q^a\right)
	\]
	many such $N, r_{\ell}, \epsilon_{\ell}$ for every base $q_i$.\\
	
	Second, let $I=[y_k+af_i,x_k)$ for some $2\leq k\leq k_t$ such that $k\notin \{k_1,\dots, k_{t}\}$. Further let $i=1$ if $k<k_1$, otherwise let $i$ be such that $k_{i-1}<k<k_{i}$. By construction, we have that
	\[
		(M_kd_i)^{[y_k,x_k)} = 1.
	\]
	Hence, we have
	\[
		(r_kM_kd_i)^{[y_k,x_k)} \leq R = q_1^{\rho}.
	\]
	Again, we want to find the number of $N< n\leq 2N$, $r_k$, $r_{\ell}\in\{1,2,\dots,R\}$ and $\epsilon_{\ell}\in\{0,1\}$ for $\ell \in \{1,\dots, Q\}\setminus\{k\}$ such that in base $q_i$
	\[
		\left(\left(n+\sum_{\substack{\ell = 1\\\ell\neq k}}^{Q}\varepsilon_{\ell}r_{\ell}M_{\ell}\right)d_i+a_i\right)^{I}
				\neq
		\left(\left(n+\sum_{\substack{\ell = 1\\\ell\neq k}}^{Q}\varepsilon_{\ell}r_{\ell}M_{\ell}\right)d_i+a_i+r_{k}M_{k}d_i\right)^{I}.
	\]	
	We have that the only possibility such that these two values are not equal is if
	\[
		\left(\left(n+\sum_{\substack{\ell = 1\\\ell\neq k}}^{Q}\varepsilon_{\ell}r_{\ell}M_{\ell}\right) d_i +a_i\right)^{[y_k+\rho+1,x_k)} = q_i^{x_k-y_k-\rho-1}-1.
	\]
	Therefore, by fixing all $\epsilon_i$ and all $r_i$, we see by second part of Lemma \ref{odd-elimination} that there are at most
	\[
		O\left(\frac{N}{q_i^{x_k-y_k-\rho-1}}\right)
	\]
	such numbers $n$. We conclude that the size of the exceptional set is at most
	\[
		O\left(R^Q\frac{N}{q_i^{x_k-y_k-\rho-1}}\right).
	\]
	
	Third, let $I=[\rho_i',y_{k_i-1}+af_i)$, for some positive integer $a+2\rho>\rho_1'>a+\rho$ to be chosen later. Further, let $\rho_i'= \left\lfloor(\rho_1')\frac{\log(q_1)}{\log(q_t)}\right\rfloor$. By the termination condition in the recursive definition of the $y_k$ we can deduce that the interval $I$ is non-empty. In base $q_i$, we have that
	\[
		(r_{k_i}M_{k_i}d_i)^{[0,y_{k_i-1}+af_i)} \leq Rq_1^{a} = q_1^{a+\rho}.
	\]
	Similarly to above, the only way that
	\[
		\left(\left(n+\sum_{\substack{\ell = 1\\\ell\neq k_i}}^{Q}\varepsilon_{\ell}r_{\ell}M_{\ell}\right)d_i+a_i\right)^{I}
				\neq
		\left(\left(n+\sum_{\substack{\ell = 1\\\ell\neq k_i}}^{Q}\varepsilon_{\ell}r_{\ell}M_{\ell}\right)d_i+a_i+r_{k_i}M_{k_i}d_i\right)^{I}
	\]	
	is that
	\[
		\left(\left(n+\sum_{\substack{\ell = 1\\\ell\neq k_i}}^{Q}\varepsilon_{\ell}r_{\ell}M_{\ell}\right) d_i +a_i\right)^{\Big[\left\lceil (a+\rho) \frac{\log(q_1)}{\log(q_t)} \right\rceil,\rho'_i\Big)} = q_i^{\rho'_i-\left\lceil (a+\rho) \frac{\log(q_1)}{\log(q_t)} \right\rceil}-1\;.
	\]
	Therefore, by fixing again all $\epsilon_i$ and all $r_i$, we see that because 
    \[
        N>q_i^{y_{k_i-1}+af_i}
    \]
    that
    \[
		\left(\left(n+\sum_{\substack{\ell = 1\\\ell\neq k_i}}^{Q}\varepsilon_{\ell}r_{\ell}M_{\ell}\right) d_i +a_i\right)^{\Big[\left\lceil (a+\rho) \frac{\log(q_1)}{\log(q_t)} \right\rceil,\rho'_i\Big)}
	\]
    attains every value between 0 and $q_i^{\rho'_i-\left\lceil (a+\rho) \frac{\log(q_1)}{\log(q_t)} \right\rceil}-1$ equally often. Hence, we have that there are at most
	\[
		O\left(\frac{N}{q_i^{\rho'_i-\left\lceil (a+\rho) \frac{\log(q_1)}{\log(q_t)} \right\rceil}}\right) = O\left(\frac{N}{q_1^{\rho'_1-a-\rho}}\right)
	\]
	many numbers $n$ such that
    	\[
		\left(\left(n+\sum_{\substack{\ell = 1\\\ell\neq k_i}}^{Q}\varepsilon_{\ell}r_{\ell}M_{\ell}\right) d_i +a_i\right)^{\Big[\left\lceil (a+\rho) \frac{\log(q_1)}{\log(q_t)} \right\rceil,\rho'_i\Big)} = q_i^{\rho'_i-\left\lceil (a+\rho) \frac{\log(q_1)}{\log(q_t)} \right\rceil}-1\;.
	\]
    Therefore, we conclude that the size of the exceptional set is at most
	\[
		O\left(R^Q\frac{N}{q_1^{\rho'_1-a-\rho}}\right).
	\]
	For the forth case, let $I=[0,a)$. Note that $M_k\equiv0\bmod q_1^a$, and hence we have for every $r_k$ that
	\[
		\left(r_kM_kd_1\right)^{I}=0.
	\]
	We deduce that for every combination of the $\epsilon_{\ell}$'s the value of
	\[
		\left(\left(n+\sum_{\ell = 2}^{Q}\varepsilon_{\ell}r_{\ell}M_{\ell}\right) d_i +a_i\right)^{I}
	\]
	is constant, therefore the exceptional set is empty.\\
	
	Having defined and analysed the intervals $I$ above, we conclude that in each case, we can find (at least) one index $k$ such that
	\[
		\left(\left(n+\sum_{\ell = 2}^{Q}\varepsilon_{\ell}r_{\ell}M_{\ell}\right)d_i+a_i\right)^{I}
			\neq
		\left(\left(n+\sum_{\ell = 2}^{Q}\varepsilon_{\ell}r_{\ell}M_{\ell}\right)d_i+a_i+r_1M_1d_i\right)^{I}
	\]
	except for at most
	\begin{align*}
		E_2 &= O\left(R^QN\left(\frac{1}{q^a}+\frac{q^a}{N}+\frac{1}{q_i^{x_k-y_k-\rho-1}}+\frac{1}{q_1^{\rho'_1-a-\rho}}\right)\right)= O\left(R^QN\left(\frac{1}{q_1^{\rho'_1-a-\rho}}\right)\right)
	\end{align*}
		
	many combinations of the $r_i$ and $n$.\\
	
	We now recall our observations preceding equation \eqref{eq_observation} and we see that the digits contained in one of the above intervals do not contribute to $K$. We conclude that
	\begin{align*}
		K(r_1M_1,\dots, r_QM_Q) = \frac 1N \sum_{N<n\leq 2N} \sum_{\varepsilon\in \{0,1\}^Q}\mathcal{C}^{|\varepsilon|}
		\e\Bigg(&\theta_1 s_{q_1}^{[a, \rho'_1)}\left(\left(n+\sum_{\ell = 1}^{Q} \varepsilon_{\ell}r_{\ell}M_{\ell}\right)d_1 			+a_1\right)\\
			&+\sum_{i=2}^{t}\theta_i s_{q_i}^{[0, \rho'_i)} \left(\left(n+\sum_{\ell = 1}^{Q} \varepsilon_{\ell}r_{\ell}M_{\ell}\right)d_i+a_i\right)\Bigg)\\
		+ O\left(R^Q\left(\frac{1}{q_1^{\rho'_1-a-\rho}}\right)\right).&
	\end{align*}
	
	We now observe that the function $s_{q_i}^{[0, \rho'_i)}(n)$ only depends on the last  ${\rho'_i}$. By the assumption that $\eta$ is small enough, we deduce by Corollary \ref{cor_independence} that 
	\begin{align*}
		\hspace{-10pt}K(r_1M_1,\dots, r_QM_{Q}) &= \frac 1{q_1^{\rho'_1-a}} \sum_{1\leq n \leq q_1^{\rho'_1-a}} \sum_{\varepsilon\in \{0,1\}^Q}\mathcal{C}^{|\varepsilon|}
		\e\left(\theta_1 s_{q_1}^{[a,\rho'_1)}\left(\left(nq_1^{a}+\sum_{\ell=1}^{Q}\varepsilon_{\ell}r_{\ell}M_{\ell}\right)d_1 +a_1\right)\right)\\ 
		&\hspace{-55pt}\times \frac 1{\lcm(q_2^{\rho'_2},\dots,q_t^{\rho'_t})} \sum_{n=0}^{\lcm(q_2^{\rho'_2},\dots,q_t^{\rho'_t})} \sum_{\varepsilon\in \{0,1\}^Q}\mathcal{C}^{|\varepsilon|}
		\e\left(\sum_{i=2}^{t}\theta_i s_{q_i}^{[0,\rho'_i)}\left(\left(n+\sum_{\ell=1}^{Q}\varepsilon_{\ell}r_{\ell}M_{\ell}\right)d_i +a_i\right)\right)\\
		& + O\left(R^Q\left(\frac{1}{q_1^{\rho_1'-a-\rho}}\right) +  \frac{1}{q_1^{\rho_1}}\right).
	\end{align*}
	The second error term comes from Corollary \ref{cor_independence} and the definition of $\rho$. Furthermore, we appeal to the relationship of $\rho$ and $\rho_1'$ to see that the second error term is dominated by the first, hence we can ignore it in the big $\mathrm{O}$ notation.\\
	
	We further observe that $|S(\theta_1, \dots, \theta_t)|^{2^Q}$ only depends on $|K(r_1M_1,\dots, r_QM_{Q})|$. We can therefore estimate the second factor trivially with 1 and we note that by assumption $q_1^a\mid M_{\ell}$ for all $1\leq \ell\leq Q$ we can find $M_{\ell}'$ such that $q_1^aM_{\ell}'= M_{\ell}$. Combining these two observations yields
	\begin{align*}
		|K(r_1M_1,\dots, r_QM_{Q})| &\leq \left|\frac 1{q_1^{\rho'-a}} \sum_{1\leq n \leq q_1^{\rho'-a}} \sum_{\varepsilon\in \{0,1\}^Q}\mathcal{C}^{|\varepsilon|} \e\left(\theta_1 s_{q_1}^{[a,\rho'_1)}\left(\left(nq_1^{a}+\sum_{\ell=1}^{Q}\varepsilon_{\ell}r_{\ell}M_{\ell}'q_1^{a}\right)d_1+a_1\right)\right)\right|\\
		& + O\left(R^Q\left(\frac{1}{q_1^{\rho_1'-a-\rho}}\right)\right).
	\end{align*}
	We plug this estimate back into \eqref{eq_spower} and get
	\begin{align*}
		S(\theta_1, \dots, \theta_t)|^{2^Q} &\ll \frac {1}{R^Q}\sum_{r\in\{1,\dots,R\}^Q}\\ &\left|\frac 1{q_1^{\rho_1'-a}} \sum_{1\leq n \leq q_1^{\rho'_1-a}} \sum_{\varepsilon\in \{0,1\}^Q}\mathcal{C}^{|\varepsilon|} \e\left(\theta_1 s_{q_1}^{[a,\rho_1')}\left(\left(nq_1^{a}+\sum_{\ell=1}^{Q}\varepsilon_{\ell}r_{\ell}M_{\ell}'q_1^{a}\right)d_1+a_1\right)\right)\right| + E_3,
	\end{align*}
	where
	\begin{equation*}
		E_3 = O\left(\frac{(M_1+\dots + M_{Q})R}{N}+\frac 1R + \frac{1}{q_1^{\rho_1'-a-\rho}}\right) =O\left(\frac{1}{q_1^{\rho_1'-a-\rho}}\right),
	\end{equation*}
	because we know that $R=q_1^{\rho}$ and that $M_k\leq q_1^{28a}$.
	We extend the summation of $R$ to $r\in\{1,\dots,q_1^{\rho'_1-a}\}^Q$ and see that
	\begin{align*}
		|S(\theta_1, \dots, \theta_t)|^{2^Q} &\ll \frac {1}{q_1^{\rho'_1-a}R^Q}\sum_{r\in\{1,\dots,q_1^{\rho_1'-a}\}^Q}\\
		 &\hspace{10pt}\left|\sum_{1\leq n \leq q_1^{\rho'_1-a}} \sum_{\varepsilon\in \{0,1\}^Q}\mathcal{C}^{|\varepsilon|} \e\left(\theta_1 s_{q_1}^{[a,\rho_1')}\left(\left(nq_1^{a}+\sum_{\ell=1}^{Q}\varepsilon_{\ell}r_{\ell}M_{\ell}'q_1^{a}\right)d_1+a_1\right)\right)\right|+E_3.
	\end{align*}
	
	Next, we observe that because $\gcd(d_1,q_1)=1$, $nq_1^{a}d_1$ attains every digit combination in the interval $[a,\rho'_1)$ exactly once as $n$ varies from $n=1$ to $n=q_1^{\rho'_1-a}$. Therefore, we see that the shift by $a_1$ is just a reordering of this list. Hence, we can ignore the summand $a_1$. Similarly, if we replace $nq_1^{a}d_1$ by just $nq_1^{a}$, we see that this also just corresponds to a reordering of the digit combination.\\
	
	Additionally, we have that $\gcd(M_{\ell}'d_1, q_1)=1$, because we ensured that $\delta_a(M_{\ell})\neq 0$, i.e. $\delta_1(M_{\ell}')\neq 0$. By the same observation as above, we see that $r_{\ell}M_{\ell}'q_1^{a}d_1$ attains every combination of the digits at positions $[a,\rho'_1)$ exactly once, as $r_{\ell}\in \{1,\dots,q_1^{\rho'_1-a}\}$. Similarly to above, we observe that we can replace $r_{\ell}M_{\ell}'q_{a}d_1$ by $r_{\ell}q_1^{a}$.\\
	Combining these observations yields
	\begin{align*}
		|S(\theta_1, \dots, \theta_t)|^{2^Q} &\ll \frac {1}{q_1^{\rho_1'-a}R^Q} \sum_{r\in\{1,\dots,q_1^{\rho_1'-a}\}^Q} \\
		&\hspace{10pt}\left|\sum_{1\leq n \leq q_1^{\rho_1'-a}} \sum_{\varepsilon\in \{0,1\}^Q}\mathcal{C}^{|\varepsilon|} \e\left(\theta_1 s_{q_1}^{[a,\rho_1')}\left(nq_1^{a}+\sum_{\ell=1}^{Q}\varepsilon_{\ell}r_{\ell}q_1^{a}\right)\right)\right|+ O\left(\frac{1}{q_1^{\rho'_1-a-\rho}}\right).
	\end{align*}
	By the triangle inequality we can pull out the sum over $n$ and since $s_{q_1}^{[a,\infty)}\left(nq_1^{a}\right)=s_{q_1}\left(n\right)$, we have
	\begin{align*}
		|S(\theta_1, \dots, \theta_t)|^{2^Q} &\ll \frac {1}{q_1^{\rho_1'-a}R^Q} \sum_{1\leq n \leq q_1^{\rho_1'-a}} \sum_{r\in\{1,\dots,q_1^{\rho'}\}^Q}\left| \sum_{\varepsilon\in \{0,1\}^Q}\mathcal{C}^{|\varepsilon|} \e\left(\theta_1 s_{q_1}^{[0,\rho'-a)}\left(n+\sum_{\ell=1}^{Q}\varepsilon_{\ell}r_{\ell}\right)\right)\right| \\
		&+ O\left(\frac{1}{q^{\rho'_1-a-\rho}}\right).
	\end{align*}
	Squaring the equation above and applying the Cauchy-Schwarz inequality, we can get rid of the absolute value and get only one additional summation variable $r_{Q+1}$, compare \cite{DrmotaSpiegelhofer2025}.
	\begin{align*}
		|S(\theta_1, \dots, \theta_t)|^{2^{Q+1}} &\ll \frac {1}{q_1^{\rho_1'-a}R^{Q+1}} \sum_{1\leq n \leq q_1^{\rho'_1-a}}  \sum_{r\in\{1,\dots,q_1^{\rho'}\}^{Q+1}} \sum_{\varepsilon\in \{0,1\}^{Q+1}}\mathcal{C}^{|\varepsilon|}
		\e\left(\theta_1 s_{q_1}^{[0,\rho'_1-a)}\left(n+\sum_{\ell=1}^{Q+1}\varepsilon_{\ell}r_{\ell}\right)\right)\\ 
		&+ O\left(\frac{1}{q_1^{\rho_1'-a-\rho}} \right).
	\end{align*}
	We obtain additional error terms of $q_1^{\rho}/N $ and $1$ which are both dominated by $q_1^{a+\rho-\rho_1'}$. The right-hand side is now nearly a Gowers norm. It is only off by a factor
	\[
		\frac{q_1^{(\rho_1'-a)(Q+1)}}{R^{Q+1}}
	\]
	Therefore, we conclude that by Corollary 1.6 in Jelinek \cite{Jelinek2025} the right-hand side is bounded above by
	\begin{equation*}
		\ll \frac{q_1^{(\rho_1'- a)(Q+1)}}{R^{Q+1}} N^{-c||(q_1-1)\theta_1||^2} + O\left(\frac{1}{q_1^{\rho_1'-a-\rho}} \right)= q_1^{(\rho_1' - a - \rho)(Q+1)} N^{-c||(q_1-1)\theta_1||^2} + O\left(\frac{1}{q_1^{\rho'_1-a-\rho}} \right).
	\end{equation*}
	If we choose 
	\[
		\rho_1'=\min\{a+2\rho, a+\rho + c||(q_1-1)\theta_1||^2/(Q+2)\},
	\]
	we deduce that the expression above is bounded by
	\begin{equation*}
		 \ll N^{-c'||(q_1-1)\theta_1||^2}
	\end{equation*} Therefore we get that
	\begin{equation}
		|S(\theta_1, \dots, \theta_t)|^{2^{Q+1}} \ll N^{-c'||(q_1-1)\theta_1||^2}
	\end{equation}
	and in particular that
	\begin{equation}
		|S(\theta_1, \dots, \theta_t)| \ll N^{-c''||(q_1-1)\theta_1||^2},
	\end{equation}
	proving the proposition.
\end{proof}
\section{Proof of Theorem \ref{main_thm_indep}}\label{sec_proof_main_theorem}
Now we are able to proof Theorem \ref{main_thm_indep}
\begin{proof}[Proof of Theorem~\ref{main_thm_indep}]
	We first need to find an arithmetic progression $dk+L$, with step size $d$ and shift $L$. Let $d=q_1^{\xi_1+\zeta_1}q_2^{\xi_2+\zeta_2}$. By the heuristic we expect that
	\begin{align}
		\mathbb{E}(s_{q_1}(dk+L))&=\frac{q_1-1}{2}(\log_{q_1}N - \zeta_1 - \xi_1) + s_{q_1}^{[0,\xi_1+\zeta_1)}(L)\\
		\mathbb{E}(s_{q_2}(dk+L))&=\frac{q_2-1}{2}(\log_{q_2}N - \zeta_2 - \xi_2) + s_{q_2}^{[0,\xi_2+\zeta_2)}(L)
	\end{align}
	Therefore if we infer that $\mathbb{E}(f(dk+L))=b\mathbb{E}(s_{q_1}(dk+L))-a\mathbb{E}(s_{q_2}(dk+L))=0$, we get that $\zeta_1$ and $\zeta_2$ and $L$ need to satisfy the following relation:
	\begin{align*}
	    \zeta_2 =  \Biggl\lfloor log_{q_2}&(N)\left(1-\frac ba \frac{(q_1-1)\log q_2}{(q_2-1)\log q_1}\right)+\frac ba \frac{q_1-1}{q_2-1}(\xi_1+\zeta_1)-\xi_2\\
        &+\frac{2}{q_2-1}\left(s_{q_2}^{[0,\xi_2+\zeta_2)}(L')-\frac ba s_{q_1}^{[0,\xi_1+\zeta_1)}(L')\right)\Biggr\rfloor
	\end{align*}
	We pick $\zeta_1 = 0$, and $\zeta_2 = (1-\varepsilon)\log_{q_2}(N)$. It follows that by the assumption $\xi_1,\xi_2=o(\log N)$, we have $s_{q_1}^{(\xi_1+\zeta_1)}(L)=o(\log N)$. Hence, we are left with
	\[
	(1-\varepsilon)\log_{q_2}(N) \sim log_{q_2}(N)\left(1-\frac ba \frac{(q_1-1)\log q_2}{(q_2-1)\log q_1}\right) + \frac{2}{q_2-1}s_{q_2}^{\xi_2+\zeta_2}(L)
	\]
	and that
	\begin{align*}
	s_{q_2}^{[0,\xi_2+\zeta_2)}(L) &\sim \log_{q_2}(N)\left(\frac{q_2-1}{2}-\varepsilon-\frac{q_2-1}{2}+\frac ba \frac{(q_1-1)\log q_2}{2\log q_1}\right)\\
	&= \log_{q_2}(N)\left(\frac ba \frac{(q_1-1)\log q_2}{2\log q_1}-\varepsilon\right) < \log_{q_2} (N) \frac{q_2-1}{2},
	\end{align*}
	by our assumption that
	\[
		\frac ba \frac{(q_1-1)\log q_2}{(q_2-1)\log q_1}<1
	\]
	We pick $L$ such that $L\equiv L' \bmod q_1^{\xi_1}q_2^{\xi_2}$ and $L < d$. Therefore, we have that the value of the sum of digits function of $L$ can be anything between $o(\log N)$ and $(1-\varepsilon)(q_2-1)\log_{q_2}(N)$, we have in particular that we can find an L such that the equation above is satisfied.
	
	We also note that the step size satisfies
	\[
	d = q_1^{\xi_1+\zeta_1}q_2^{\xi_2+\zeta_2}<N^{1-\varepsilon}
	\]
	hence $d$ is indeed the step size of an arithmetic progression in the dyadic interval $[N,2N]$. We now want to use Proposition \ref{thm_concentration} and Proposition \ref{prop_equi} to show that there at least $N^{\gamma}$ many integers $n\in [N,2N]$ such that $f(n)=jm$ for some $-J\leq j\leq J$. Afterwards Proposition \ref{prop_shift} guarantees that $f(n+d_{-j})=0$.\\
	
	In order to apply Proposition \ref{thm_concentration} and Proposition \ref{prop_equi} it remains to show that $d$ is $N^{\eta}$-admissible for any $\eta>0$, if $N$ is sufficiently large.
	
	We choose $\zeta_1',\zeta_2'$ such that $\lfloor\log\log N\rfloor|\xi_1+\zeta_1'$ and $\lfloor\log\log N\rfloor|\xi_2+\zeta_2'$. Further we assume $|\zeta_i'-\zeta_i|\leq \log\log N$ for $i=1,2$. This changes the value of the sum of digit function by at most $O(\log\log N)$, hence does not influence the analysis of the relationship between the variables above. We now apply Corollary \ref{cor_powers_admissible} to 
	\[
	d'=q_1^{\xi_1+\zeta_1'}q_2^{\xi_2+\zeta_2'}=\left(q_1^{(\xi_1+\zeta_1')/\lfloor\log\log N\rfloor}q_2^{(\xi_2+\zeta_2')/\lfloor\log\log N\rfloor}\right)^{\lfloor\log\log N\rfloor}.
	\]
	This gives if $N$ is sufficiently large that $d'$ is $N^{\eta}$-admissible for all $\eta>0$. Hence, we can apply Proposition \ref{thm_concentration} and Proposition \ref{prop_equi} to the arithmetic progression $d'k+L$. Let
	\begin{equation}
		A'':=\left(L'+q_1^{\xi_1+\zeta_1'}q_2^{\xi_2+\zeta_2'}\mathbb{N}\right)\cap [N,2N)
	\end{equation}
	and
	\begin{equation}
	I:=\left\{k\in\mathbb{N}:N\leq L'+q_1^{\xi_1+\zeta_1'}q_2^{\xi_2+\zeta_2'} k< 2N\right\}.
\end{equation}
	We have that $|f(n)|\leq Jm$ for all $n\in A''$ except at most $C|I|\lambda^{-D}$.	Proposition \ref{prop_equi} gives us that the values of $f(n)$ are equidistributed modulo $m$ for each such $L'$. Hence, there are at least $|I|/m(1+o(1))$ many $n\in A''$ such that $f(n)\equiv 0 \bmod m$. Combining those two observations we see that, assuming that $D>3$, we have for each $L'$ at least
	\[
	\frac{C|I|}{m}(1+o(1))
	\]
	many $n\in A''\subset [N,2N]$ such that $|f(n)|\leq Jm$ and $f(n)\equiv 0 \bmod m$. In particular this states that $f(n)=jm$ for some $-J\leq j\leq J$.\\
    We note that modulo $r=\gcd(a(q_1-1,b(q_2-1))$, we have that
    \[
        f(n) \equiv as_{q_1}(n)-bs_{q_2}(n) \equiv (a-b)n \bmod r 
    \]
    Hence, recalling that $m\equiv 1\bmod r$ we see that if there exist an $n$ such that $f(n)=jm$ then $\gcd(a-b,r)\mid \gcd(j,r)$, which is the condition on $j$ in Proposition \ref{prop_shift}. Now we can apply Proposition \ref{prop_shift} and we see that 
	\[
	   f(n+d_{-j})=jm-jm=0.
	\]
	For the quantitative statement, we note that since $s_{q_2}(L') \sim  \frac{(q_1-1)}{2r\log_{q_1}}\log(N)$, we have at least $N^{\beta-\varepsilon}$ many such $L'$. Additionally, we see that $|I|\sim N^{\varepsilon'}$. Hence, combining these two observations we get that there are at least
	\[
	   N^{\beta-\varepsilon}N^{\varepsilon'}=N^{\beta-\varepsilon''}
	\]
	many integers in $[N,2N]$ such that
	\[
	   f(n+d_{-j})=jm-jm=0.
	\]
\end{proof}
\section{Proof of Theorem \ref{main_thm_dep}}\label{sec_Proof_of_dep}
In this section we prove Theorem \ref{main_thm_dep}, which describes which ratios of two sum of digits functions are possible, if the two bases are multiplicatively dependent. The methods involved are very explicit and independent of the other sections of this paper.
\begin{proof}[Proof of Theorem~\ref{main_thm_dep}]
	The key to the proof is the following observation:\\
	Since $q_1=b^k$ and $q_2=b^{\ell}$ we see that if we group the digits from the right in blocks of $k$ digits in base $b$, each such block corresponds to a digit in base $q_1$. Similarly, if we group the digits from the right in blocks of $\ell$ digits in base $b$, each such block corresponds to a digit in base $q_2$.\\
	Our first aim is to show that $\frac{1}{b^{\ell-1}}$ and $b^{k-1}$ can each be attained infinitely many times. For this observe the following identities for the digit sums of $b^x$ for some $x\in \NN$, and $a < b$ a non-negative integer:
	\[
	s_{q_1}(ab^x)=ab^{x\bmod k} \text{ and } s_{q_2}(ab^x)=ab^{x\bmod \ell}
	\]
	We can now use the Chinese remainder theorem by our assumption that $k$ and $\ell$ are coprime, to find infinitely many $x$, such that $x\equiv k-1 \bmod k$ and $x~\equiv 0 \bmod \ell$, which gives for any $a< b$ that
	\[
	\frac{s_{q_1}(ab^x)}{s_{q_2}(ab^x)}=\frac{ab^{k-1}}{ab^0}=b^{k-1}.
	\]
	Similarly, we can find infinitely many $x$ such that $x \equiv 0 \bmod k$ and $x\equiv \ell -1 \bmod \ell$ which corresponds to the ratio
	\[
	\frac{s_{q_1}(ab^x)}{s_{q_2}(ab^x)}=\frac{1}{b^{\ell-1}}.
	\]
	To see that these values are indeed the maximum and the minimum, respectively, we write $n=\sum_{x\in \NN}a_xb^x$ in the base $b$ expansion. By the fact that $k$ digits in base $b$ correspond to a single digit in base $q_1$ we have that 
	\[
	s_{q_1}(n) = s_{q_1}\left(\sum_{x\in \NN}a_xb^x\right) = \sum_{x\in \NN}a_xs_{q_1}(b^x)=\sum_{x\in \NN}a_xb^{x\bmod k}.
	\]
	Similarly, in base $q_2$ we have that
	\[
	s_{q_2}(n) =\sum_{x\in \NN}a_xb^{x\bmod \ell}.
	\]
	Therefore, we have that
	\begin{align*}
		\frac{s_{q_1}(n)}{s_{q_2}(n)} &= \frac{\sum_{x\in \NN}a_xb^{x\bmod k}}{\sum_{x\in \NN}a_xb^{x\bmod \ell}}.
	\end{align*}
	One can easily show by induction that
	\begin{align*}
		\min_{\substack{x\in\NN\\a_x\neq 0}}\left(\frac{b^{x\bmod k}}{b^{x\bmod \ell}}\right) \leq \frac{\sum_{x\in \NN}a_xb^{x\bmod k}}{\sum_{x\in \NN}a_xb^{x\bmod \ell}} \leq \max_{\substack{x\in\NN\\a_x\neq 0}}\left(\frac{b^{x\bmod k}}{b^{x\bmod \ell}}\right).
	\end{align*}
	Therefore, we see that 
	\[ 
	\frac{1}{b^{\ell-1}}\leq r \leq b^{k-1}.
	\]
	It only remains to prove that for any such rational number $r=u/v$, with $u,v$ coprime, we can find an $n\in \NN$ such that
	\[
	\frac{s_{q_1}(n)}{s_{q_2}(n)}=\frac uv
	\]
	Let $x>0$ be minimal such that $x\equiv k-1 \bmod k$ and $x\equiv 0 \bmod \ell$, and let $x_i=x+ik\ell$. Similarly, let $y>0$ be minimal such that $y\equiv 0 \bmod k$ and $y\equiv \ell-1 \bmod \ell$, and let $y_i=y+ik\ell$. Consider
	\[
	n=\sum_{i=1}^{s}b^{x_i}+\sum_{i=s+1}^{s+t}b^{y_i}
	\]
	for some $s,t>0$. Our observations above give that $s_{q_1}(n)= sb^{k-1}+tb^{0}$ and $s_{q_2}(n)= sb^{0}+tb^{\ell-1}$. Therefore, we are trying to solve the following equation for $s,t>0$:
	\[
	\frac{sb^{k-1}+t}{s+tb^{\ell-1}}=\frac uv
	\]
	Rearranging gives that
	\[
	s(vb^{k-1}-u)=t(ub^{\ell-1}-v)\,.
	\]
	We see that the brackets are both positive if (and only if)
	\[ 
	\frac{1}{b^{\ell-1}}\leq \frac uv \leq b^{k-1}\,.
	\]
	Therefore, for any such $u,v$ the equation has infinitely many solutions for $s,t>0$. Hence, recalling that $n=\sum_{i=1}^{s}b^{x_i}+\sum_{i=s+1}^{s+t}b^{y_i}$ we see that
	\[
	\frac{s_{q_1}(n)}{s_{q_2}(n)}=\frac uv.
	\]
	To get a quantitative lower bound, we consider 
	\[
	n=a+\sum_{i=r}^{r+s}b^{x_i}+\sum_{i=r+s+1}^{s+r+t}b^{y_i}
	\]
	where $r$ is minimal such that $a<b^{x_i}$. We say that $s_{q_1}(a)=e$ and $s_{q_2}(a)=f$ and we again want to solve for $s,t>0$ such that
	\[
	\frac{sb^{k-1}+t+e}{s+tb^{\ell-1}+f}=\frac uv.
	\]
	Rearranging the equation above, we get that
	\[
	s(vb^{k-1}-u)-t(ub^{\ell-1}-v)=f(ub^{\ell-1}-v)-e(vb^{k-1}-u)\,.
	\]
	If we assume that $a\leq b^M$, we have that $e,f\leq cM$, for some constant $c$. Therefore, we can find a solution for $s,t$ to the above equation if $s,t$ are of the size $\sim cM$ and hence $n\leq b^{c_1M}$. 
	
	In particular, this shows that given all integers up to $N$, we have at least $N^c$ many numbers $n$ such that
	\[\frac{s_{q_1}(n)}{s_{q_2}(n)}=\frac uv.\]
\end{proof}

%\section{Acknowledgements}

\printbibliography

@misc{BloomCroot2025,
      title={Integers with small digits in multiple bases}, 
      author={Thomas F. Bloom and Ernie Croot},
      year={2025},
      eprint={2509.02835},
      archivePrefix={arXiv},
      primaryClass={math.NT},
      url={https://arxiv.org/abs/2509.02835}, 
}

@Article{BretecheStollTenenbaum2019,
 Author = {R\'egis {de la Bret\`eche} and Thomas {Stoll} and G\'erald {Tenenbaum}},
 Title = {{Somme des chiffres et changement de base}},
 FJournal = {{Annales de l'Institut Fourier}},
 Journal = {{Ann. Inst. Fourier}},
 ISSN = {0373-0956; 1777-5310/e},
 Volume = {69},
 Number = {6},
 Pages = {2507--2518},
 Year = {2019},
 Publisher = {Universit\'e Joseph Fourier, Grenoble; Association des Annales de l'Institut Fourier, Saint-Martin d'H\`eres},
 Language = {French},
 MSC2010 = {11A63 11K16 11K60 11J82},
 Zbl = {1455.11018}
}

@incollection{DHLL2017,
  TITLE = {Sums of the digits in bases $2$ and $3$},
  AUTHOR = {Deshouillers, Jean-Marc and Habsieger, Laurent and Laishram, Shanta and Landreau, Bernard},
  URL = {https://hal.archives-ouvertes.fr/hal-01401869},
  BOOKTITLE = {Number theory --- {D}iophantine problems, uniform distribution and applications},
  PUBLISHER = {Springer},
  PAGES = {211-217},
  YEAR = {2017},
  HAL_ID = {hal-01401869},
  HAL_VERSION = {v1},
}

@misc{DimitrovHowe2021,
      title={Powers of $3$ with few nonzero bits and a conjecture of {E}rd{\H{o}}s},
      author={Vassil S. Dimitrov and Everett W. Howe},
      year={2021},
      eprint={2105.06440},
      archivePrefix={arXiv},
      primaryClass={math.NT}
}

@misc{DrmotaSpiegelhofer2025,
      title={The joint distribution of binary and ternary digits sums}, 
      author={Michael Drmota and Lukas Spiegelhofer},
      year={2025},
      eprint={2501.00850},
      archivePrefix={arXiv},
      primaryClass={math.NT},
      url={https://arxiv.org/abs/2501.00850}, 
}

@Article{Erdos1979,
 Author = {Paul {Erd\H{o}s}},
 Title = {{Some unconventional problems in number theory}},
 FJournal = {{Mathematics Magazine}},
 Journal = {{Math. Mag.}},
 ISSN = {0025-570X; 1930-0980/e},
 Volume = {52},
 Pages = {67--70},
 Year = {1979},
 Publisher = {{Mathematical Association of America (MAA), Washington, D.C.; Taylor \& Francis, Abingdon, Oxfordshire}},
 Language = {English},
 MSC2010 = {11-02 11A05 11A25 11N05 00A07},
 Zbl = {0407.10001}
}

@Book{Erdos1980,
 Author = {Paul {Erd\H{o}s} and Ronald L. {Graham}},
 Title = {{Old and new problems and results in combinatorial number theory}},
 FJournal = {{Monographies de l'Enseignement Math\'ematique}},
 Journal = {{Monogr. Enseign. Math.}},
 ISSN = {0425-0818},
 Volume = {28},
 Year = {1980},
 Publisher = {L'Enseignement Math\'ematique, Universit\'e de Gen\`eve, Gen\`eve},
 Language = {English},
 MSC2010 = {11-02 00A07 11Bxx},
 Zbl = {0434.10001}
}

@Misc{Furstenberg1970,
 Author = {H. {Furstenberg}},
 Title = {{Intersections of Cantor sets and transversality of semi-groups}},
 Year = {1970},
 Language = {English},
 HowPublished = {{Probl. Analysis, Sympos. in Honor of Salomon Bochner, Princeton Univ. 1969, 41-59 (1970).}},
 MSC2010 = {28A78 28A12},
 Zbl = {0208.32203}
}

@Article{GranvilleRamare1996,
 Author = {Andrew {Granville} and Olivier {Ramar\'e}},
 Title = {{Explicit bounds on exponential sums and the scarcity of squarefree binomial coefficients}},
 FJournal = {{Mathematika}},
 Journal = {{Mathematika}},
 ISSN = {0025-5793; 2041-7942/e},
 Volume = {43},
 Number = {1},
 Pages = {73--107},
 Year = {1996},
 Publisher = {John Wiley \& Sons, Chichester; London Mathematical Society, London},
 Language = {English},
 MSC2010 = {11B65 11L20 11L07},
 Zbl = {0868.11009}
}

@misc{Jelinek2025,
      title={Gowers norms for linearly recurrent numeration systems}, 
      author={Pascal Jelinek},
      year={2025},
      eprint={2510.16947},
      archivePrefix={arXiv},
      primaryClass={math.NT},
      url={https://arxiv.org/abs/2510.16947}, 
}

@Article{MauduitRivat2010,
AUTHOR = {Mauduit, Christian and Rivat, Jo\"{e}l},
     TITLE = {Sur un probl\`eme de {G}elfond: la somme des chiffres des
              nombres premiers},
   JOURNAL = {Ann. of Math. (2)},
  FJOURNAL = {Annals of Mathematics. Second Series},
    VOLUME = {171},
      YEAR = {2010},
    NUMBER = {3},
     PAGES = {1591--1646},
      ISSN = {0003-486X},
   MRCLASS = {11K36 (11A63 11J71 11L07 11N13)},
  MRNUMBER = {2680394},
MRREVIEWER = {Y.-F. S. P\'{e}termann},
       DOI = {10.4007/annals.2010.171.1591},
       URL = {https://doi.org/10.4007/annals.2010.171.1591},
}

@Article{Sarkozy1985,
 Author = {A. {S\'ark\"ozy}},
 Title = {{On divisors of binomial coefficients. I}},
 FJournal = {{Journal of Number Theory}},
 Journal = {{J. Number Theory}},
 ISSN = {0022-314X; 1096-1658/e},
 Volume = {20},
 Pages = {70--80},
 Year = {1985},
 Publisher = {Elsevier (Academic Press), San Diego, CA},
 Language = {English},
 MSC2010 = {11B65 05A10 11L03},
 Zbl = {0551.10002}
}

@Article{SengeStraus1973,
 Author = {H. G. {Senge} and E. G. {Straus}},
 Title = {{PV-numbers and sets of multiplicity}},
 FJournal = {{Periodica Mathematica Hungarica}},
 Journal = {{Period. Math. Hung.}},
 ISSN = {0031-5303; 1588-2829/e},
 Volume = {3},
 Pages = {93--100},
 Year = {1973},
 Publisher = {Springer Netherlands, Dordrecht; Akad\'emiai Kiad\'o, Budapest},
 Language = {English},
 MSC2010 = {11R06},
 Zbl = {0248.12004}
}

@Article{Shmerkin2019,
 Author = {Pablo {Shmerkin}},
 Title = {{On Furstenberg's intersection conjecture, self-similar measures, and the $L^q$ norms of convolutions}},
 FJournal = {{Annals of Mathematics. Second Series}},
 Journal = {{Ann. Math. (2)}},
 ISSN = {0003-486X; 1939-8980/e},
 Volume = {189},
 Number = {2},
 Pages = {319--391},
 Year = {2019},
 Publisher = {{Princeton University, Mathematics Department, Princeton, NJ}},
 Language = {English},
 MSC2010 = {11K55 28A80 37C45},
 Zbl = {1426.11079}
}

@article {Spiegelhofer2018,
    AUTHOR = {Spiegelhofer, Lukas},
     TITLE = {Pseudorandomness of the {O}strowski sum-of-digits function},
   JOURNAL = {J. Th\'eor. Nombres Bordeaux},
  FJOURNAL = {Journal de Th\'eorie des Nombres de Bordeaux},
    VOLUME = {30},
      YEAR = {2018},
    NUMBER = {2},
     PAGES = {637--649},
      ISSN = {1246-7405,2118-8572},
   MRCLASS = {11A55 (11A67)},
  MRNUMBER = {3891330},
MRREVIEWER = {Vilius\ Stakenas},
       URL = {http://jtnb.cedram.org/item?id=JTNB_2018__30_2_637_0},
}

@misc{Spiegelhofer2023,
      title={Thue--Morse along the sequence of cubes}, 
      author={Lukas Spiegelhofer},
      year={2023},
      eprint={2308.09498},
      archivePrefix={arXiv},
      primaryClass={math.NT},
      url={https://arxiv.org/abs/2308.09498}, 
}

@article {Spiegelhofer2023a,
    AUTHOR = {Spiegelhofer, Lukas},
     TITLE = {Collisions of digit sums in bases 2 and 3},
   JOURNAL = {Israel J. Math.},
  FJOURNAL = {Israel Journal of Mathematics},
    VOLUME = {258},
      YEAR = {2023},
    NUMBER = {1},
     PAGES = {475--502},
      ISSN = {0021-2172,1565-8511},
   MRCLASS = {11A63 (11B25 11B50 60F10)},
  MRNUMBER = {4682942},
MRREVIEWER = {Clemens\ Heuberger},
       DOI = {10.1007/s11856-023-2478-8},
       URL = {https://doi.org/10.1007/s11856-023-2478-8},
}

@article{Stewart1980,
 Author = {C. L. {Stewart}},
 Title = {On the representation of an integer in two different bases},
 FJournal = {Journal f\"ur die {R}eine und {A}ngewandte {M}athematik},
 Journal = {J. {R}eine {A}ngew. {M}ath.},
 ISSN = {0075-4102; 1435-5345/e},
 Volume = {319},
 Pages = {63--72},
 Year = {1980},
 Publisher = {De Gruyter, Berlin},
 Language = {English},
 MSC2010 = {11A63 11A25},
 Zbl = {0426.10008}
}

@misc{Toumi2025,
      title={The level of distribution of the sum-of-digits function in arithmetic progressions}, 
      author={Nathan Toumi},
      year={2025},
      eprint={2504.02784},
      archivePrefix={arXiv},
      primaryClass={math.NT},
      url={https://arxiv.org/abs/2504.02784}, 
}

@Article{Wu2019,
 Author = {Meng {Wu}},
 Title = {{A proof of Furstenberg's conjecture on the intersections of $\times p$- and $\times q$-invariant sets}},
 FJournal = {{Annals of Mathematics. Second Series}},
 Journal = {{Ann. Math. (2)}},
 ISSN = {0003-486X; 1939-8980/e},
 Volume = {189},
 Number = {3},
 Pages = {707--751},
 Year = {2019},
 Publisher = {{Princeton University, Mathematics Department, Princeton, NJ}},
 Language = {English},
 MSC2010 = {11K55 28A50 28A80 28D05 37C45},
 Zbl = {1430.11106}
}
\end{document}